\documentclass[11pt]{article}
\usepackage{latexsym, amsfonts, amscd, amssymb, verbatim, amsmath,
enumerate}

\newtheorem{theorem}{Theorem}[section]
\newtheorem{proposition}{Proposition}[section]
\newtheorem{definition}{Definition}[section]

\newtheorem{lemma}{Lemma}[section]

\newtheorem{corollary}{Corollary}[section]

\newtheorem{example}{Example}[section]

\newcommand{\proofbox}{\hspace{\fill}{$\Box$}}
\newenvironment{proof}{\textbf{Proof}.}{\proofbox}

\def\bar{\overline}
\date{}

\numberwithin{equation}{section}

\usepackage{booktabs}

\usepackage{tikz}
\usepackage{pgfplots}
\pgfplotsset{compat=1.17}

\definecolor{mycolor1}{rgb}{0,0.4470,0.7410}

\usepackage{graphicx}

\usepackage[a4paper, left=2.7cm, right=2.3cm, top=3cm, bottom=2.5cm]{geometry}

\usepackage{tabularx}

\usepackage[linesnumbered,ruled,vlined]{algorithm2e}

\begin{document}

\title{On space-time FEM for time-optimal control problems governed by parabolic equations with mixed  and endpoints constraints} 
	
\author{H. Khanh\footnote{Department of Mathematics, FPT University, Hoa Lac Hi-Tech Park, Hanoi, Vietnam; email: khanhh@fe.edu.vn}  \   and      B.T. Kien\footnote{Department of Optimization and Scientific Computing, Institute of Mathematics, Vietnam Academy of Science and Technology,  18 Hoang Quoc Viet road, Hanoi, Vietnam; email: btkien@math.ac.vn}}
	
\maketitle
	
\medskip

\noindent {\bf Abstract.} {\small In this paper,  we consider a class of time-optimal control problems governed by linear parabolic equations with mixed control-state constraints and end-point constraints, and without Tikhonov regularization term in the objective function.   By the finite element method, we discretize the optimal control problem  to obtain a sequence of mathematical programming problems in finite-dimensional spaces. Under certain conditions, we show that the optimal solutions of the discrete problems converge to an optimal solution of the original problem. Besides, we show that if the second-order sufficient condition is satisfied, then some error estimates of approximate solutions are obtained.

\medskip

	

\section{Introduction}
	
Let $\Omega$ be a bounded domain in $\mathbb{R}^N$ with $N=2$ or $3$ and its  boundary $\Gamma=\partial \Omega$ is of class $C^2$. Let $D=H^2(\Omega)\cap H_0^1(\Omega)$, $V=H_0^1(\Omega)$,  $H=L^2(\Omega)$ and $Q_T:=\Omega\times (0, T)$, $\Sigma_T := \Gamma\times [0, T]$. We consider  the problem of finding $T>0$, a control function $w  \in L^\infty(Q_T)$  and the corresponding state $y\in C(\bar Q_T)\cap W^{1, 1}_2(0, T; D, H)$ which solve
\begin{align}
    &J(T,y, w) :=   T     \to  \min   \label{P1}\\
    &{\rm s.t.}\notag\\
 (P) \quad  \quad  &\frac{\partial y}{\partial t}  -   \Delta y   = w \quad \text{in} \ Q_T, \quad y(x, t)=0 \quad \text{on}  \ \Sigma_T,    \label{P2} \\
    &y(0)=y_0  \quad \text{in} \ \Omega,      \label{P3}\\
    &H(y(T))  \le  0,  \label{P4} \\
    &a \le w(x, t) + \alpha y(x, t) \le b \quad \text{a.a.} \ (x, t)\in Q_T,  \label{P5}
\end{align}
 where $a, b, \alpha \in \mathbb R$ with $a < b$, $\alpha \ge 0$;  $y_0\in H^1_0(\Omega)\cap C(\bar\Omega)$  
 and the terminal constraint function $H$ is defined as 
 \begin{align}
     H: L^2(\Omega) \to \mathbb R, \quad   \zeta \mapsto H(\zeta) := \frac{1}{2}\int_\Omega (\zeta(x) - y_\Omega(x))^2 dx  -   \frac{\lambda^2}{2}, 
 \end{align}
 with $y_\Omega \in H_0^1(\Omega)$ and $\lambda > 0$ are given problem data.   Throughout of the paper, we denote by    $\Phi$ the  feasible set of problem $(P)$, that is, $\Phi$ consists of triples $(T, y, v) \in (0, +\infty)  \times \big(C(\bar Q_T)\cap W^{1, 1}_2(0, T; D, H)\big) \times L^\infty(Q_T)$ which satisfies conditions \eqref{P2}-\eqref{P5}. The objective of the problem $(P)$ is to steer the heat equation from an initial state $y_0$ into a ball of radius $\lambda$, centered at the desired state $y_\Omega$, in the shortest possible time $T > 0$, while ensuring that the mixed pointwise constraints (\ref{P5}) are satisfied.

 The class of time-optimal control problems is important for modern technology and has been studying  by some authors recently. For the papers which have a close connection with the present work, we refer the readers to \cite{Arada-2003}, \cite{Bon}, \cite{Bon1}, \cite{Bon2}  \cite{KKR-2024-3} and information therein.   Among these papers, L. Bonifacius et al. \cite{Bon} gave a priori estimates for space time-time finite element discretization of parabolic time-optimal problems with Tikhonov regularization term in the objective function, with end-point constraint and control constraints. Meanwhile  L. Bonifacius et al. \cite{Bon1} gave a priori estimates for space time finite element discretization of parabolic time-optimal problems with  bang-bang controls with end-point constraint and control constraints. Particularly,  L. Bonifacius and K. Kunish \cite{Bon2} established an equivalence of time-optimal control problems and distance-optimal control problems  for a class of parabolic control systems. Based on this equivalence, they gave an efficient algorithm to find 
solution of time-optimal control problems. 

 In this paper we consider problem $(P)$ which is  governed by linear parabolic equations with mixed control-state constraints and end-point constraints, and without Tikhonov regularization term in the objective function. The aim of this paper is to prove the convergence and give some error estimates of approximate solution to  discrete parabolic time-optimal problems which are discretized by the space-time finite element  method. To obtain the goal, we first transform $(P)$ into problem $(P')$ with pure control constaint as follows: 
 we define a control variable
 \begin{align}
     u  = w + \alpha y 
 \end{align}
 and rewrite  $(P)$ in the form of time-optimal control problem  with pure control constraint 
 \begin{align}
    &J(T,y, u) =  T      \to  \inf   \label{P1.0}\\
    &{\rm s.t.}\notag\\
 (P') \quad \quad  \quad   &\frac{\partial y}{\partial t}  -   \Delta y    +  \alpha y   = u \quad \text{in} \ Q_T, \quad y(x, t)=0 \quad \text{on}  \ \Sigma_T,    \label{P2.0} \\
    &y(0)=y_0  \quad \text{in} \ \Omega,      \label{P3.0}\\
    &H(y(T))  \le  0,  \label{P4.0} \\
    &a \le u(x, t)\le b \quad \text{a.a.} \ (x, t)\in Q_T.  \label{P5.0}
\end{align}
Denote $\Phi'$ by the  feasible set of $(P')$. The second step is to transform $(P')$ into a problem $(P'')$ with fixed-end time. We then establish the existence of optimal solution and optimality conditions for locally optimal solutions to  $(P'')$. The third step is to discretize $(P'')$  by the  space-time finite element  method to obtain a sequence of optimization problems $(P''_\sigma)$ in finite-dimensional spaces. We then show that  under certain conditions of the original problem, the optimal solutions of the discrete problems $(P''_\sigma)$ converge to an optimal solution of the original problem and that  if the second-order sufficient condition is satisfied, then some error estimates of approximate solutions are obtained.

It is noted that there is very little literature on second-order sufficient optimality conditions for parabolic time-optimal control problems where the Tikhonov term appears in the functional $J$ (see \cite{Bon}, \cite{KKR-2024-3}). And the method of proof used in these cases does not work for problems with objective functional depending only on $T$.  Another issue that has been interested in studying of  researchers recently  is to derive second- order optimality conditions for locally  optimal solutions instead of locally strong optimal solution (see Definition \ref{locally-strongly-optimal-solutions}-(ii)). In this paper, we present a second-order condition for locally  optimal solutions of problem $(P'')$ based on the usual extended critical cone. This makes a difference with the results in this direction proved in \cite{Bon1}, where a nodal set condition was required.

The remainder of the paper is organized as follows. In Section 2, we establish some related results on the existence  of optimal  solutions to problem $(P'')$ and necessary and sufficient  optimality conditions for locally optimal solution to $(P'')$.  In Section 3, we present  main results on the convergence and error estimate of approximate solutions to discrete optimization problem $(P''_\sigma)$. In order to illustrate the obtained results, a numerical example is given in Section 3. 
 
\section{Some related results for continuous problem}

\subsection{State equation}

In the sequel, we assume that $\partial\Omega$ is of class $C^2$. Let  $H:=L^2(\Omega)$, $V:=H_0^1(\Omega)$ and $D:= H^2(\Omega)\cap H_0^1(\Omega)$. The norm and the scalar product in $H$ are denoted by $|\cdot|$ and $(\cdot, \cdot)_H$,  respectively.   It is known that 
$$
D\hookrightarrow\ \hookrightarrow V\hookrightarrow\ \hookrightarrow H
$$ and each space is dense in the following one. Hereafter,  $\hookrightarrow\ \hookrightarrow$ denotes a compact embedding. 

\noindent $W^{m,p}(\Omega)$ for $m$ integer and $p\geq 1$ is a Banach space of elements in $ v\in L^p(\Omega)$ such that their generalized derivatives $D^\alpha v\in L^p(\Omega)$ for $|\alpha|\leq m$. The norm of a element  $v\in W^{m,p}(\Omega)$ is given by 
$$
\|u\|_{m,p}=\sum_{|\alpha|\leq m}\|D^\alpha u\|_p,
$$ where $\|\cdot\|_p$ is the norm in $L^p(\Omega)$. The closure of $C_0^\infty(\Omega)$ in  $W^{m,p}(\Omega)$ is denoted by $W_0^{m, p}(\Omega)$.   When $p=2$, we write $H^m(\Omega)$ and $H^m_0(\Omega)$ for $W^{m, 2}(\Omega)$ and $W^{m, 2}_0(\Omega)$, respectively. 

\noindent $W^{s, p}(\Omega)$ for real number $s\geq 0$ consists of elements  $v$ so that 
$$
\|v\|_{s,p}=\Big[\|v\|^p_{m,p} +\sum_{|\alpha|=m}\int_\Omega\int_\Omega \frac{|D^\alpha v(x)-D^\alpha v(x')|^p}{|x-x'|^{N+\sigma p}} dxdx' \Big]^{1/p}< +\infty,
$$ where $s=m+\sigma$ with $m=[s]$ and $\sigma\in (0, 1)$. It is known that 
\begin{align}
  W^{s,p}(\Omega)=\big(W^{k_0, p}(\Omega), W^{k_1, p}(\Omega)\big)_{\theta, 1},  
\end{align} where $s=(1-\theta)k_0 +\theta k_1$, $0<\theta <1$, $k_0<k_1$ and $1\leq p\leq \infty$. Here $\big(W^{k_0, p}(\Omega), W^{k_1, p}(\Omega)\big)_{\theta, 1}$ is a interpolation space in $K-$method. 

\noindent $W^{2l, l}_p(Q_T)$ for $l$ integer and $p\geq 1$ is a Banach space of elements in $ v\in L^p(Q_T)$ such that their generalized derivatives $D^r_t D^s_x v\in L^p(\Omega)$ with $2r+s\leq 2l$. The norm of a element  $v\in W^{2l,l}_p(\Omega)$ is given by 
\begin{align*}
    \big\| v\big\|_{p, Q_T}^{(2l)}=\sum_{j=0}^{2l} \langle\langle v\rangle\rangle^{(j)}_{p, Q_T},    \quad {\rm with} \ \langle\langle v\rangle\rangle^{(j)}_{p, Q_T}  = \sum_{2l+s=j}\|D^r_t D^s_x v\|_{L^p(Q_T)}.  
\end{align*}   

Given $\alpha\in (0, 1)$ and $T>0$, we denote by $C^{0, \alpha}(\Omega)$ and $C(\bar Q_T)$ the space of H\"{o}lder continuous functions on $\Omega$ and the space of continuous functions on $\bar Q_T$, respectively.    Let $H^{-1}(\Omega)$ be the dual of $H_0^1(\Omega)$. We define the following function spaces
\begin{align*}
& H^1(Q_T)=W^{1,1}_2(Q_T)=\{y\in L^2(Q_T):  \frac{\partial y}{\partial x_i}, \frac{\partial y}{\partial t}\in L^2(Q_T)\},\\
&V_2(Q_T)=L^\infty(0, T; H)\cap L^2(0, T; V),\\
&W(0, T)=\{y\in L^2(0, T; H^1_0(\Omega)): y_t\in L^2(0, T; H^{-1}(\Omega))\},\\
& W^{1,1}_2(0, T; V, H)=\{y\in L^2([0, T], V): \frac{\partial y}{\partial t}\in L^2([0, T], H) \},\\
& W^{1,1}_2(0, T; D, H)=\{y\in L^2([0, T], D): \frac{\partial y}{\partial t}\in L^2([0, T], H) \},\\
 &U_T=L^\infty(Q_T),\quad  Y_T =\big\{y\in  W^{1,1}_2(0, T; D, H) \cap  W^{2,1}_p(Q_T)| \Delta y\in L^p(Q_T)\big\}.
\end{align*} Hereafter, we assume that 
\begin{align} \label{Dimension1}
    1+\frac{N}{2}<p< N+2.
\end{align} Then $Y_T$ is a Banach space under the graph norm
\begin{align*}
    \|y\|_{Y_T}:= \|y\|_{W^{1, 1}_2(0, T; D, H)} + \|y\|_{W^{2,1}_p(Q_T)} +\|\Delta y\|_{L^p(Q_T)}. 
\end{align*} Note that if $y\in W^{2,1}_p(Q_T)$, then $y_t\in L^p(0, T; L^p(\Omega))$. By  the proof of Lemma \ref{Lemma-stateEq}, we have
\begin{align}\label{keyEmbed1}
  W^{2,1}_p(Q_T)\hookrightarrow\  \hookrightarrow C(\bar Q_T),  \quad  Y_T \hookrightarrow\ \hookrightarrow  C(\bar Q_T).
\end{align} 
Also, we have the following continuous embeddings: 
\begin{align}\label{keyEmbed3}
W^{1, 1}_2(0, T; D, H)\hookrightarrow C([0, T], V),\quad W(0, T)\hookrightarrow C([0, T], H),
\end{align} where $C([0, T], X)$ stands for the space of continuous mappings $v: [0, 1]\to X$ with $X$ is a Banach space. 

Given  $y_0\in H$ and $u\in L^2(0, T; H)$,  a function $y\in W(0, T)$ is said to be a weak solution of the semilinear parabolic equation \eqref{P2.0}-\eqref{P3.0} if 
\begin{align}\label{WeakSol}
\begin{cases}
\langle y_t, v\rangle + \sum_{i,j=1}^n \displaystyle \int_\Omega D_iy D_j v dx +(\alpha y, v)_H =(u, v)_H\quad \forall v\in H_0^1(\Omega)\quad \text{a.a.}\quad t\in [0, T]\\
y(0)=y_0.
\end{cases}
\end{align} 
If a  weak solution $y$ such that $y\in W^{1,1}_2(0, T; D, H)$ then we say $y$ is a {\it strong solution} of \eqref{P2.0}-\eqref{P3.0}. From now on a  solution to \eqref{P2.0}-\eqref{P3.0} is understood a strong solution.  

Let us make the following assumption which is related to the state equation. 

\begin{enumerate}

\item[$(H1)$] $y_0\in H_0^1(\Omega)\cap W^{2-\frac{2}p, p}(\Omega)$, where $p>1$ satisfying \eqref{Dimension1}. 

\end{enumerate}

Since $(2-\frac{2}p)-\frac{N}p\in(0, 1)$, Theorem 1.4.3.1 and Theorem 1.4.4.1 in \cite{Grisvard} imply that $W^{2-\frac{2}p, p}(\Omega)\hookrightarrow C^{0, \alpha}(\bar \Omega)$ with $\alpha := (2-\frac{2}p)-\frac{N}p.$

 The proof of the next result can be found in \cite[Lemma 2.2]{KKR-2024-3} and \cite[Theorem 2.1]{Casas-2020}.
 
\begin{lemma}\label{Lemma-stateEq} Suppose that $(H1)$ is valid. Then for each $T>0$ and  $u\in L^p(Q_T)$ with $p$ satisfying \eqref{Dimension1},  the state equation \eqref{P2.0}-\eqref{P3.0} has a unique  solution   $y\in Y_T$ and there exists positive constant $C > 0$ such that 
\begin{align}\label{KeyInq0}
    \|y\|_{C(\bar Q_T)}  \le C (\|u\|_{L^p(Q_T)} + \|y_0\|_{L^\infty(\Omega)}).
\end{align} 
Moreover, if  $u_n \rightharpoonup u$ weakly in $L^p(Q_T)$ then $y_n  \to y$ in $L^\infty(Q_T)$ and $y_n(\cdot, T)  \to y(\cdot, T)$ in $L^\infty(\Omega)$, where $y_n = y(T, u_n)$ and $y = y(T, u)$.  
\end{lemma}

\subsection{Optimal control problem with fixed end-time}

Given a triple $(T, y, q)\in \Phi$ (or $(T, y, u)\in \Phi'$), we define the extension of $y$ on the right by setting 
\begin{align*}
    y_e(x, t)=
    \begin{cases}
        y(x, t)\quad &\text{if}\quad (x, t)\in\bar\Omega\times [0, T]\\
        y(x, T)\quad  &\text{if}\quad (x, t)\in\bar\Omega\times (T,  +\infty). 
    \end{cases}
\end{align*} Then $y_e$ is continuous on $\bar\Omega\times [0, +\infty)$. Moreover, for each $T_0\in (T, +\infty)$, $y_e$ is uniform continuous on the compact set $\bar\Omega\times[0, T_0]$. 

\begin{definition}\label{Def-LocalOptim}   A triple $(T_*,  y_*,  q_*) \in \Phi$ is said to be a locally optimal solution of $(P)$ if  there exists some $\varepsilon > 0$ such that 
    \begin{align}\label{StrongOptimDef}
     T_* \le  T,     &\quad  \forall (T, y, q)\in \Phi\ \text{satisfying}\notag \\ 
     & \quad  |T -  T_*| +  \max_{(x, t) \in \bar \Omega \times [0, T_*\vee T]} |y_e(x, t) -  y_{*e}(x, t)| + \mathop{\rm esssup}\limits_{(x,t)\in\Omega\times [0, T_*]}|q(x, \frac{Tt}{T_*})-  q_*(x, t)| < \varepsilon. 
    \end{align} Hereafter $T_*\vee T=\max (T_*, T)$ and $Q_{T_*} = \Omega \times (0, T_*)$.   Moreover, if $T_* \le   T$ for all $(T, y, q)\in \Phi$ 
   we say  $(T_*,  y_*,  q_*)$ is a globally optimal solution of $(P)$. 
\end{definition} It is noted that when $T=T_*$, Definition \ref{Def-LocalOptim}  becomes a definition for optimal solutions to the problem with fixed final time $T_*$. {\it Locally optimal solutions of $(P')$ are defined similarly to Definition \ref{Def-LocalOptim}.}

\begin{proposition}
\label{E1}
    $(T_*,  y_*,  q_*) \in \Phi$ is a locally optimal solution to $(P)$ if and only if $(T_*,  y_*,  u_*) \in \Phi'$ with $u_*  = q_* + \alpha y_*$  is a locally optimal solution to $(P')$. 
\end{proposition}
\begin{proof}
    $(\Rightarrow)$  Assume that $(T_*,  y_*,  q_*) \in \Phi$ is a locally optimal solution to $(P)$. Let $\varepsilon > 0$ such that (\ref{StrongOptimDef}) is valid. Fix a number $T_0> T_*$.  By the uniform continuity of $y_{*e}$ on $\bar\Omega\times [0,  T_0]$, there exists $\delta \in (0, \frac{\varepsilon}{4{\rm max}(1, \alpha)})$ such that 
\begin{align}
\label{01}
| y_{*e}(x, t_1)-y_{*e}(x, t_2)|\le \frac{\epsilon}{4{\rm max}(1, \alpha)}  \quad  \forall t_1, t_2\in [0,  T_0] \ \  \text{satisfying}\ \  |t_1-t_2| < \delta. 
\end{align}
Define $u_*  = q_* + \alpha y_*$, we see that $(T_*,  y_*,  u_*) \in \Phi'$. Let $(T, y, u) \in \Phi'$ such that 
 \begin{align}
 \label{02}
    |T -  T_*| +  \max_{(x, t) \in \bar \Omega \times [0, T_*\vee T]} |y_e(x, t) -  y_{*e}(x, t)| + \mathop{\rm esssup}\limits_{(x,t)\in\Omega\times [0, T_*]}|u(x, \frac{Tt}{T_*})-  u_*(x, t)| < {\rm min}(\delta, T_0 - T_*). 
    \end{align} 
    Putting $q  = u  - \alpha y$, then it's easy to check that $(T, y, q) \in \Phi$. By (\ref{02}) we have 
    \begin{align}
        &  |T -  T_*|, \  \mathop{\rm esssup}\limits_{(x,t)\in\Omega\times [0, T_*]}|u(x, \frac{Tt}{T_*})-  u_*(x, t)| < \delta  \le  \frac{\varepsilon}{4{\rm max}(1, \alpha)}  \le  \frac{\varepsilon}{4}   \label{03} \\ 
        {\rm and} \quad & |T -  T_*| +  \max_{(x, t) \in \bar \Omega \times [0, T_*\vee T]} |y_e(x, t) -  y_{*e}(x, t)|  < \delta \le \frac{\varepsilon}{4{\rm max}(1, \alpha)}  \le  \frac{\varepsilon}{4}.  \label{04}
    \end{align}
    Without loss of generality, we may assume that $T \ge T_*$, then  we have  $[0, T\vee T_*]=[0, T]$. We have for $(x, t)\in\Omega\times [0, T_*]$,   
    \begin{align}
    \label{05}
        |q(x, \frac{Tt}{T_*})-  q_*(x, t)| \le |u(x, \frac{Tt}{T_*})-  u_*(x, t)| +  \alpha |y(x, \frac{Tt}{T_*})-  y_*(x, t)|
    \end{align}
    Also, for $(x, t)\in\Omega\times [0, T_*]$,
    \begin{align}
    \label{06}
        |y(x, \frac{Tt}{T_*})-  y_*(x, t)|   \le |y(x, \frac{Tt}{T_*})  -  y_{*e}(x, \frac{Tt}{T_*})|  +  | y_{*e}(x, \frac{Tt}{T_*})  -  y_*(x, t)|,  
    \end{align}
    \begin{align}
        |y(x, \frac{Tt}{T_*})  -  y_{*e}(x, \frac{Tt}{T_*})|  &\le   \max_{(x, t) \in \bar \Omega \times [0, T_*]} |y(x, \frac{Tt}{T_*})  -  y_{*e}(x, \frac{Tt}{T_*})|  \nonumber \\ 
        &\le   \max_{(x, t) \in \bar \Omega \times [0, T]} |y(x,  t)  -  y_{*e}(x,  t)| <  \frac{\varepsilon}{4{\rm max}(1, \alpha)}.  \label{07}
    \end{align}
    We have that for all $t \in [0, T_*]$, $\frac{t}{T_*} \in [0, 1]$ and hence, $|\frac{Tt}{T_*}  -  t| = \frac{t}{T_*}|T  - T_*| \le |T - T_*| \le \delta$. It follows from (\ref{01}) that 
    \begin{align}
        \label{08}
        | y_{*e}(x, \frac{Tt}{T_*})  -  y_*(x, t)|  \le \frac{\varepsilon}{4{\rm max}(1, \alpha)}.
    \end{align}
    From (\ref{06}), (\ref{07}) and (\ref{08}) we obtain 
    \begin{align}
        \label{09}
        |y(x, \frac{Tt}{T_*})-  y_*(x, t)|   < \frac{\varepsilon}{4{\rm max}(1, \alpha)} + \frac{\varepsilon}{4{\rm max}(1, \alpha)} = \frac{\varepsilon}{2{\rm max}(1, \alpha)}. 
    \end{align}
    From (\ref{05}), (\ref{03}) and (\ref{09}) we obtain
    \begin{align}
        \label{010}
        |q(x, \frac{Tt}{T_*})-  q_*(x, t)| < \frac{\varepsilon}{4}   + \alpha \frac{\varepsilon}{2{\rm max}(1, \alpha)}  \le \frac{3 \varepsilon}{4}. 
    \end{align}
    Combining this with (\ref{04}),  we have
    \begin{align}
        |T -  T_*| +  \max_{(x, t) \in \bar \Omega \times [0, T_*\vee T]} |y_e(x, t) -  y_{*e}(x, t)|   + \mathop{\rm esssup}\limits_{(x,t)\in\Omega\times [0, T_*]}|q(x, \frac{Tt}{T_*})-  q_*(x, t)| < \varepsilon.
    \end{align}
    It follows from this and (\ref{StrongOptimDef}) that $T \ge T_*$. Thus, $(T_*,  y_*,  u_*)$ is a locally optimal solution to $(P')$. 

     $(\Leftarrow)$   Conversely, if $(T_*,  y_*,  u_*) \in \Phi'$  is a locally optimal solution to $(P')$. By the same argument as above, we can show that $(T_*,  y_*, q_*)  =  (T_*,  y_*,  u_* - \alpha y_*)$ is a locally optimal solution to $(P)$.

\end{proof}

We set  $Q_1=\Omega\times(0, 1)$, $\Sigma_1 = \Gamma\times [0, 1]$ and 
\begin{align*}
    U_1=  L^\infty(Q_1),   \
    Y_1= \Big\{\zeta \in W^{1, 1}_2(0, 1; D, H) \cap W^{2,1}_p(Q_1)| \  \Delta \zeta \in L^p(Q_1)\Big\}.
\end{align*}
Then $Y_1$ is a Banach space under the graph norm $\|\zeta\|_{Y_1}:= \|\zeta\|_{W^{1, 1}_2(0, 1; D, H)} + \|\zeta\|_{W^{2,1}_p(Q_1)} +   \big\|\Delta \zeta\big\|_{L^p(Q_1)}$.   By changing variable 
\begin{align*}
t= Ts, \quad s \in [0, 1]
\end{align*} 
and $\zeta(x, s) = y(x, Ts)$, $v(x, s) = u(x, Ts)$ in $(P')$,  we obtain the following  optimal control problem with fixed final time: 
\begin{align}
    &J(T,  \zeta, v)    =  T   \to  \inf    \label{P1.1}\\
    &{\rm s.t.}   \notag\\
    (P'') \quad  \quad \quad &\dfrac{\partial \zeta}{\partial s}  -   T \Delta \zeta +  \alpha T \zeta  = Tv   \quad \text{in} \ Q_1, \quad \zeta(x, s)=0 \quad  \text{on} \  \Sigma_1,   \label{P2.1}\\
    &\zeta(0)=y_0 \quad \text{in} \ \Omega,    \label{P3.1}\\
    &H(\zeta(1))  \le   0,      \label{P4.1}\\
    &a \le   v(x, s) \le  b \quad {\rm a.a.} \ x \in \Omega, \ \forall s  \in [0, 1].   \label{P5.1}
\end{align}
Define $Z =   \mathbb{R}\times Y_1\times U_1$ and denote $\Phi''$ by  the  feasible set of problem $(P'')$, that is, $\Phi''$ consists of triples $(T, \zeta, v) \in Z$ which satisfies conditions \eqref{P2.1}-\eqref{P5.1}.  Given a vector $z=(T, \zeta, v)\in Z$, 
\begin{align*}
\|z\|_*:=  |T| +\|\zeta\|_{C(\bar Q_1)}+\|v\|_{U_1},\ \
\|z\|_Z:=  |T| +\|\zeta\|_{Y_1}+\|v\|_{U_1}. 
\end{align*} Given $z_0\in Z$ and $r>0$, we denote by $B_{*Z}(z_0,r)$ and $B_Z(z_0, r)$ the balls center at $z_0$ and radius $r$ in norm $\|\cdot\|_*$ and $\|\cdot\|_Z$, respectively. 

\begin{definition}  
\label{locally-strongly-optimal-solutions}
$(i)$ \   Vector $z_*=  (T_*, \zeta_*, v_*)\in \Phi''$ is a locally optimal solution to $(P'')$ if there exists $\delta>0$ such that 
\begin{align}
\label{DfLocalOptimSol}
T_*  \le T, \quad   \quad \forall   z = (T, \zeta, v) \in \Phi'' \ \text{satisfying}\ \|z-z_*\|_*   <  \delta.  
\end{align}  

$(ii)$ \   An element  $z_*=  (T_*, \zeta_*, v_*)\in \Phi''$ is said to be a locally strongly optimal solution of the problem $(P'')$ if there exists $\delta>0$ such that 
\begin{align}
\label{DfLocalOptimSol.strong}
T_*  \le T, \quad   \quad \forall   z = (T, \zeta, v) \in \Phi'' \ \text{satisfying}\  |T  -  T_*|   <  \delta.  
\end{align} 

$(iii)$ \ If $T_*  \le T$ for all $(T, \zeta, v) \in \Phi''$  we say  $(T_*,  y_*,  v_*)$ is a globally optimal solution of $(P'')$.    
\end{definition}
Since $Y_1\hookrightarrow C(\bar Q_1)$, there exists $\gamma>0$ such that $\|z\|_*\leq \gamma \|z\|_Z$ for all $z\in Z$.  Therefore,  if $\|z-z_*\|_Z  <   \frac{\delta}{\gamma}$, then $\|z-z_*\|_*   <  \delta$. Consequently,  if $z_*$ is a locally optimal solution to $(P'')$ in norm $\|\cdot\|_*$, then it is also locally optimal solution to $(P'')$  in norm $\|\cdot\|_Z$. Also, we have  $|T - T_*| \le \|z-z_*\|_*$.  This implies that if $z_*\in \Phi''$ is a locally strongly optimal solution of the problem $(P'')$ then it is a locally optimal solution of the problem $(P'')$.

The following propositions give relations between optimal solutions of $(P')$ and $(P'')$.
\begin{proposition} 
\label{E2}
\label{relationP'andP''}
$(i)$ \  Suppose that $ (T_*, y_*,  u_*)\in  \Phi'$  is a locally  optimal solution to $(P')$. Let    
    \begin{align*} 
     \zeta_* (x, s) := y_*(x, T_*s),\   v_*(x, s) := u_*(x,  T_*s)  \   {\rm for} \ s \in [0, 1]. 
\end{align*} 
Then the vector  $(T_*,  \zeta_*,  v_*)$ is a locally optimal solution to $(P'')$.  

$(ii)$ \ Conversely, if  $T_*>0$ and  $(T_*, \zeta_*, v_*)\in \Phi''$ is a locally optimal solution to $(P'')$. Define
    \begin{align*} 
   y_*(x, t):=\zeta_* (x, \frac{t}{T_*}),\   u_*(x, t):= v_*(x, \frac{t}{T_*})  \   {\rm for} \ t \in [0, T_*]. 
\end{align*} 
Then the vector $(T_*,  y_*, u_*)$  is  a locally  optimal solution to $(P')$.
\end{proposition}
\begin{proof}
$(i)$ \   By definition, $T_*>0$ and there exists $\varepsilon > 0$  such that \eqref{StrongOptimDef} is valid. Fix a number $T_0> T_*$.  By the uniform continuity of $y_{*e}$ on $\bar\Omega\times [0,  T_0]$, there exists $\delta \in (0, \frac{\varepsilon}{4})$ such that 
\begin{align}
\label{0.1}
| y_{*e}(x, t_1)-y_{*e}(x, t_2)|\le \frac{\varepsilon}{4}  \quad  \forall t_1, t_2\in [0,  T_0] \ \  \text{satisfying}\ \  |t_1-t_2| < \delta. 
\end{align} We can choose $\delta>0$ small enough such that $T>0$ whenever $|T-T_*|<\delta$.
Clearly, $(T_*,  \zeta_*, v_*) \in \Phi''$.  Let $(T, \zeta, v)\in\Phi''$ such that 
\begin{align}
\label{0.2}
  |T-T_*|   + \|\zeta-\zeta_*\|_{C(\bar Q_1)} + \|v - v_*\|_{L^\infty(\bar Q_1)}   <\min(\delta, T_0-T_*). 
\end{align}  
We define $ y(x, t)=  \zeta(x,  \frac{t}{T}),\   u(x, t)=  v(x, \frac{t}{T})$. It is easy to check that $(T, y, u) \in \Phi'$.  Without loss of generality, we may assume that $T> T_*$. Then from condition $T- T_*< T_0- T_*$, we have $T<T_0$ and  $[0, T\vee T_*]=[0, T]$. Therefore, for $(x, t)\in\Omega\times [0, T]$, we have  
\begin{align}
    \label{0.4}
    |y_e(x, t) - y_{*e}(x, t)|  &=  \Big|y(x, t) - y_{*e}\big(x, t)\big)\Big| \nonumber \\
    &\le   \Big|\zeta\big(x,  \frac{t}{T}\big) -  y_*\big(x, \frac{tT_*}{T}\big)\Big|  + 
    \Big| y_*\big(x, \frac{tT_*}{T}\big)  -  y_{*e}\big(x, t\big)  \Big|   \nonumber \\
    &=   \Big|\zeta\big(x, \frac{t}{T}\big) - \zeta_*\Big(x,  \frac{t}{T} \Big)\Big|  + 
    \Big|y_{*e}\big(x, \frac{tT_*}{T}\big)   -  y_{*e}\big(x,  t\big)  \Big|.   
\end{align}
 Note that $\frac{t}{T} \in [0, 1]$ for all $t \in [0, T]$. By (\ref{0.2}) we deduce that 
\begin{align}
\label{0.41}
    \Big|\zeta\Big(x,  \frac{t}{T}\Big) -  \zeta_*\Big(x,  \frac{t}{T} \Big)\Big| \le  \|\zeta- \zeta_*\|_{C(\bar Q_1)} \le \delta \le \frac{\varepsilon}{4}
\end{align}
for all $(x, t) \in \bar \Omega \times [0,  T]$. Moreover, for all $t \in [0,  T]$  we have $t_1 = \frac{tT_*}{T} \in [0, T_*] \subset [0, T]$, $t_2 =  t \in [0,  T]$ and $|t_1 - t_2| \le |T - T_*| \le \delta$. Therefore, we have from (\ref{0.1}) that  
\begin{align*}
    \Big|y_{*e}\big(x, \frac{tT_*}{T}\big)   -  y_{*e}\big(x,  t\big)  \Big|  \le \frac{\varepsilon}{4}\quad \forall (x, t) \in \bar \Omega \times [0, T].
\end{align*} Combining this with  (\ref{0.4}) and (\ref{0.41}), we get 
\begin{align*}    
    |y_e(x, t) -  y_{*e}(x, t)|  \le \frac{\varepsilon}{4} + \frac{\varepsilon}{4} = \frac{\varepsilon}{2}\quad \forall (x, t) \in \bar \Omega \times [0, T].
\end{align*}
 This implies that $ 
    \mathop {\rm max}\limits_{(x, t) \in \bar \Omega \times [0, T\vee T_*]} |y_e(x, t) - y_{*e}(x, t)|  \le \frac{\varepsilon}{2}$.  Hence 
\begin{align} \label{0.9}
    |T - T_*|  +  \mathop {\rm max}\limits_{(x, t) \in \bar \Omega \times [0, T\vee T_*]} |y_e(x, t) -  y_{*e}(x, t)|  \le \frac{\varepsilon}{4} +  \frac{\varepsilon}{2}  = \frac{3\varepsilon}4. 
\end{align}  On the other hand 
\begin{align*}
    {\rm esssup}_{t\in [0, T_*]}| u(\frac{Tt}{T_*})-u_*(t)|={\rm essup}_{t\in [0, T_*]}| v(x, \frac{t}{T_*})-v_*(\frac{t}{T_*})|\leq\|v-v_*\|_{L^\infty(Q_1)}\leq \delta<\frac{\epsilon}4.
\end{align*} Combining this with \eqref{0.9}, yields 
$$
|T - T_*|  +  \mathop {\rm max}\limits_{(x, t) \in \bar \Omega \times [0, T\vee T_*]} |y_e(x, t) -  y_{*e}(x, t)| +  {\rm esssup}_{t\in [0, T_*]}| u(\frac{Tt}{T_*})-u_*(t)|<\varepsilon. 
$$
From this and \eqref{StrongOptimDef}, we have $T_* \le T$. Thus $(T_*,  \zeta_*,  v_*)$ is a locally optimal solution to $(P'')$.

$(ii)$ \   Let $\delta>0$ such that \eqref{DfLocalOptimSol} is valid. Fix a number  $T_0> T_*$. Since $y_{*e}$ is uniform continuous on $\bar\Omega\times [0,  T_0]$, there exists $\varepsilon\in (0, \delta)$ such that 
$|y_{*e}(x, t)-y_{*e}(x, t')|< \frac{\delta}3$ for all $t, t'\in [0, T_0]$ satisfying $|t-t'|<\varepsilon$. We can choose $\varepsilon>0$ small enough so that $T>0$ whenever $|T-T_*|< \varepsilon$.  We now take $(T, y, u)\in\Phi'$  satisfying
$$
|T-T_*| + \max_{(x, t)\in\bar\Omega\times [0, T \vee T_*]}|y_e(x,t)- y_{*e}(x, t)| + {\rm esssup}_{(x, t)\in \Omega\times[0, T_*]}|u(x, \frac{Tt}{T_*})- u_*(x, t)| <  \min(\frac{\varepsilon}6, T_0- T_*). 
$$
Without loss of generality, we can assume that $T>T_*$. Then we have $T<T_0$.  Let $\zeta(x, s)=y(x, Ts)$ and $v(x, s)=u(x, Ts)$. Then $(T, \zeta, v)\in\Phi''$. Note that 
\begin{align}\label{Estim1}
  |T- T_*|< \min(\frac{\varepsilon}6, T_0- T_*) \le \frac{\varepsilon}6.   
\end{align} 
From this and uniform continuity of $y_{*e}$,  we have for all $(x, s)\in \bar\Omega\times [0, 1]$ that
 \begin{align*}
|\zeta(x, s)-\zeta_*(x, s)| &=|y(x, T s)-  y_*(x, T_*s)|\\
&\leq |y(x, T s)- y_{*e}(x, T s)| +|y_{*e}(x, T s)- y_{*e}(x, T_* s)|
\leq \frac{\varepsilon}6 + \frac{\delta}3\leq \frac{\delta}2.
\end{align*} This implies that 
\begin{align}\label{Estim2}
    \|\zeta-\zeta_*\|_{C(\bar Q_1)}\leq \frac{\delta}2. 
\end{align}
Also, we have 
\begin{align*}
  {\rm esssup}_{(x, s)\in\Omega\times[0, 1]}|v(x, s)- v_*(x, s)| &={\rm esssup}_{(x, s)\in\Omega\times[0, 1]}|u(x, Ts)-u_*(x, T_*s)|\\
  &={\rm esssup}_{(x, t)\in\Omega\times[0, T_*]}|u(x, \frac{Tt}{T_*})- u_*(x,t))|\leq \frac{\varepsilon}6< \frac{\delta}6. 
\end{align*} Hence 
$ \|v-v_*\|_{L^\infty(Q_1)}\leq \frac{\delta}6$.   Combining this with \eqref{Estim1} and   \eqref{Estim2}, we obtain 
$$
  \|\zeta-\zeta_*\|_{C(\bar Q_1)} + \|v-v_*\|_{L^\infty(Q_1)} +|T-T_*|<  \frac{\delta}2 +\frac{\delta}6 +\frac{\delta}6  <  \delta.
$$ Since $(T_*, \zeta_*, v_*)$ is a locally optimal solution to $(P'')$, we have  $T_*  \le T$. Thus $(T_*,  y_*, u_*)$  is  a locally  optimal solution to $(P')$.

\noindent The proof of the proposition  is complete. 
\end{proof}

\subsection{Existence of optimal solutions}

Let us define 
\begin{align}
\label{dfphi+}
    \Phi''_+ = \{(T, \zeta, v) \in \Phi'': T > 0\} 
\end{align} and impose the following assumption:
\begin{itemize}
    \item [$(H2)$]  $H(y_0) > 0$. 
\end{itemize}

Then we have 

\begin{proposition}\label{E3}
    Suppose that $(H1)$ and  $(H2)$ are valid,  and there exists $(\widetilde T, \widetilde \zeta, \widetilde v) \in \Phi''_+$. Then, the problem ${\rm min}_{(T,  \zeta, v) \in \Phi''_+} T$ has a globally optimal solution $(T_*, \zeta_*, v_*)$ which is a locally strongly optimal solution to $(P'')$. 
\end{proposition}
\begin{proof} Let $T_* :=  {\rm inf}_{(T,  \zeta, v) \in \Phi''_+} T$. Then, we have $0 \le T_* \le \widetilde T$. 
    This implies that $T_*$ is finite. By definition, there exists a sequence $\{z_n = (T_n,  \zeta_n, v_n) \}\subset \Phi''_+$  such that $T_* = {\rm lim}_{n \to + \infty} T_n$.   

    Since $\{v_n\}$ is bounded in $L^p(Q_1)$, we may assume that   $ v_n \rightharpoonup  v_*$ in $L^p(Q_1)$. Note that the set
  \begin{align}
  \label{Kp}
      K_p:= \{v \in L^p(Q_1): a \le v(x, s) \le b \ {\rm a.a.} \  (x, s) \in Q_1\}
  \end{align}
  is a closed convex set in $ L^p(Q_1)$. Hence,  it is a weakly closed set. Since $v_n \in K_p$, we must have $v_*  \in K_p$. This means $a \le v_*(x, s) \le b  \ {\rm a.a.} \  (x, s) \in Q_1$.  Hence, $v_* \in  L^\infty(Q_1)$. By Lemma \ref{Lemma-stateEq},    $\zeta_n =   \zeta(T_n, v_n)  \to \zeta(T_*, v_*) =: \zeta_*$ in $L^\infty(Q_1)$ and $\zeta_n(1)  \to \zeta_*(1)$ in $L^\infty(\Omega)$. Hence
  \begin{align}
      H(\zeta_*(1))  =  {\rm lim}_{n \to + \infty} H(\zeta_n (1)) \le 0. 
  \end{align}
 If $T_* = 0$,  then we have $\zeta_*(x, s) = y_0(x)$ for all $(x, s) \in Q_1$. Hence, $\zeta_*(1) = y_0$. Then, we have $H(y_0) = H(\zeta_*(1)) \le 0$ which contradicts $(H2)$. So we must have $T_* > 0$.
  
  In a summary, we have showed that $T_n\to   T_*$ in $\mathbb R$, $\zeta_n   \to  \zeta_*$ in $L^\infty(Q_1)$, $\zeta_n(\cdot, 1)  \to \zeta_*(\cdot, 1)$ in $L^\infty(\Omega)$, $ v_n \rightharpoonup  v_*$ in $L^p(Q_1)$, $(T_*, \zeta_*, v_*)\in \Phi''_+$ and $T_* :=  {\rm inf}_{(T,  \zeta, v) \in \Phi''_+} T$.  Thus $(T_*, \zeta_*, v_*)$ is a globally optimal solution of the problem ${\rm min}_{(T,  \zeta, v) \in \Phi''_+} T$. It is clear that $(T_*, \zeta_*, v_*)$ is a locally strongly optimal solution to $(P'')$. 
\end{proof}

\begin{corollary}
    Suppose that $(H1)$ and $(H2)$ are valid,  and $\Phi$ is nonempty. Then, the problem $(P)$ has at least one globally optimal solution. 
\end{corollary}
\begin{proof}
     The conclusion follows from Proposition \ref{E1}, \ref{E2} and \ref{E3}. 
\end{proof}

\subsection{Necessary and sufficient optimality conditions for $(P'')$}

Let us define Banach spaces
\begin{align*}
Z =   \mathbb{R}\times Y_1\times U_1, \quad E = L^p(Q_1) \times \big(H_0^1(\Omega)\cap W^{2-\frac{2}p, p}(\Omega)\big),  \quad W= L^\infty(Q_1).
\end{align*} 
Define mappings $F: Z \to E$, $H : Z \to \mathbb R$ and $G : Z \to W$ by setting
\begin{align*}
    &F(T, \zeta, v)   = \Big(F_1(T, \zeta, v), \  F_2(T, \zeta, v)\Big) 
    = \Big(\frac{\partial \zeta}{\partial s} - T\Delta\zeta + \alpha T\zeta - Tv,  \  \  \zeta(0)- y_0\Big),\\
    &H(T, \zeta, v)  =  H(\zeta(1)), \\
    &G(T, \zeta, v)  = v. 
    \end{align*}
Then $F$, $H$ and $G$ are well defined (see also \cite{KKR-2024-3}, Section 4]).   We now rewrite problem $(P'')$ in the  form
\begin{align}
    \mathop {\rm min}\limits_{(T, \zeta, v) \in Z}T \quad \quad  {\rm {subject \  to}} \quad \quad F(T, \zeta, v) = 0, \  H(T, \zeta, v) \le 0, \  G(T, \zeta, v) \in K_1, 
\end{align}
where  $K_1:=\{ u\in L^\infty(Q_1)| a\leq u(x, s)\leq b\ {\rm a.a.}\ (x, s)\in Q_1\}$. Following Vexler et al \cite{Bon1} we make the following assumption on {\it linearized Slater condition}.

\noindent $(H3)$   Let $(T_*, \zeta_*, v_*)$ be a locally optimal solution to $(P'')$ and the mapping $h(T, v) := H(T, \zeta, v)$. We assume that $-  \partial_Th(T_*, v_*)  > 0$.

In the sequel, we shall use sets
\begin{align*}
 &Q_a=\{(x, s) \in Q_1:  v_*(x, s)=a\}, \\
 &Q_b=\{(x, s) \in Q_1: v_*(x, s) =b\},\\
 &Q_{ab}=\{(x, s) \in Q_1: a<  v_*(x, s) <b\}. 
\end{align*}
The first-order optimality conditions for the problem $(P'')$ are the following.
\begin{proposition}
\label{Pro-KKT-P''}
Suppose assumptions $(H1)$-$(H3)$, $T_*>0$ and $z_*= (T_*, \zeta_*, v_*)\in\Phi''$ is a locally optimal solution to $(P'')$. Then there exist Lagrange multipliers $\mu > 0$ and
\begin{align*}
\varphi, e \in L^\infty(Q_1)  \cap W^{1, 1}_2(0, 1; D, H) \cap C([0, 1], H^1_0(\Omega))
\end{align*}
satisfying the following conditions:

\noindent $(i)$  (the adjoint equations)  
\begin{align}
\label{cm0.1} 
- \dfrac{\partial \varphi}{\partial s} - T_* \Delta\varphi + \alpha T_*\varphi = 0  \ {\rm in} \ Q_1, \quad \quad \varphi = 0 \ {\rm on} \ \Sigma_1,  \quad \quad  \varphi (1) = - \mu(\zeta_*(1) - y_\Omega);
\end{align}

\noindent $(ii)$   (optimality condition for $v_*$) 
\begin{align}
\label{cm0.3}
      -  T_*\varphi(x, s) +  e(x, s)  =  0 \quad {\rm a.a.} \ (x, s) \in Q_1; 
\end{align}

\noindent $(iii)$   (optimality condition for $T_*$) 
\begin{align}\label{cm0.4}  
    \int_{Q_1}\varphi \big(- \Delta\zeta_* +  \alpha \zeta_* -v_*\big) dxds  = -1;
\end{align}

\noindent $(iv)$ (the complementary conditions)   $H(\zeta(1))  =  0$ and
\begin{align}
\label{a0.4}
    e(x, s)
    \begin{cases}
       \leq 0\quad  {\rm a.a.}  \   (x, s)\in Q_{a}\\
    \geq 0\quad  {\rm a.a.} \ (x, s)\in Q_{b}\\
     =0 \quad {\rm a.a.} \ (x, s)\in Q_{ab}.
    \end{cases}
\end{align}
Moreover, $v_* \in L^\infty(Q_1)  \cap L^2([0, 1], W^{1, s}(\Omega)) \cap H^1([0, 1], L^2(\Omega))$ for all $1 \le s < +\infty$. 
\end{proposition}
\begin{proof}
    By \cite[Proposition 4.3]{KKR-2024-3}, there exist Lagrange multipliers $\lambda, \mu \in \mathbb R_+$, $\varphi \in L^\infty(Q_1)  \cap L^2(0, 1; H^1_0(\Omega))$ and $e \in L^\infty(Q_1)$  not all are zero and  satisfy the following conditions: 
    \begin{align}
        &- \dfrac{\partial \varphi}{\partial s} - T_* \Delta\varphi + \alpha T_*\varphi = 0  \ {\rm in} \ Q_1, \quad \quad \varphi = 0 \ {\rm on} \ \Sigma_1,  \quad \quad  \varphi (1) = - \mu(\zeta_*(1) - y_\Omega);   \label{a1} \\
        &-  T_*\varphi(x, s) +  e(x, s)  =  0 \quad {\rm a.a.} \ (x, s) \in Q_1;   \label{a2} \\
        &\int_{Q_1}\varphi \big(- \Delta\zeta_* +  \alpha \zeta_* -v_*\big) dxds + \int_\Omega\lambda dx= 0;  \label{a3} \\
        &\mu H(\zeta(1))  =  0   \quad {\rm and} \quad e(x, s)\in  N([a, b],  v_*(x, s)) \ {\rm a.a.}\ (x, s)\in Q_1.   \label{a4} 
    \end{align}
    If $\mu = 0$ then from (\ref{a1}), $\varphi = 0$. From (\ref{a2}), $e = 0$ and it follows from (\ref{a3})  that $\lambda = 0$. This is impossible. Therefore, we have $\mu >0$. Combining this with (\ref{a4}) we obtain $H(\zeta(1))  =  0$, and the condition  $e(x, s)\in  N([a, b],  v_*(x, s)) \ {\rm a.a.}\ (x, s)\in Q_1$ in (\ref{a4}) implies (\ref{a0.4}). By the condition (\ref{a3}) above and \cite[Lemma 3.2, (ii)]{Bon1}, we deduce that $\int_\Omega\lambda dx = 1$.  Since $\zeta_*(1) \in H^1_0(\Omega)$, by regularity from \cite[Theorem 5, p. 360]{Evan}, we have $\varphi \in  W^{1, 1}_2(0, 1; D, H)  \hookrightarrow C([0, 1], H^1_0(\Omega))$. By (\ref{a2}), $e = T_*\varphi \in W^{1, 1}_2(0, 1; D, H)$. Finally, by same arguments as proof of Proposition 2.3 in \cite{Neitzel-2012}, we get $v_* \in L^\infty(Q_1)  \cap L^2([0, 1], W^{1, s}(\Omega)) \cap H^1([0, 1], L^2(\Omega))$ for any $1 \le s < +\infty$.
 \end{proof}

In order to deal with second-order sufficient conditions in this case, we need, for some $\beta > 0$,  a critical cone ${\mathcal C}_2^\beta[(T_*, \zeta_*,  v_*)]$ which consists of vectors $(T, \zeta, v) \in \mathbb R \times W^{1, 2}(0, 1; D, H) \times L^2(Q_1)$ satisfying the following conditions:
\begin{itemize}
     \item [$(c_1)$]  $T \le \beta (\|v\|_{L^2(Q_1)}  +  |T|)$; 
     \item [$(c_2)$]   $\dfrac{\partial \zeta}{\partial s} +T_*\big(-\Delta \zeta +   \alpha \zeta -v\big)+ T(-\Delta \zeta_* +  \alpha \zeta_* - v_*)=0, \quad    \zeta(0)=0$;
     \item [$(c_3)$]  $v(x, s) \ge 0$  a.a. $(x, s)\in Q_a$  and  $v(x, s) \le 0$  a.a. $(x, s)\in Q_b$.
\end{itemize}
The following proposition gives second-order sufficient conditions for locally optimal solutions to $(P'')$.
\begin{theorem}
    \label{SOSC-P''}
    Suppose $z_*  = (T_*, \zeta_*, v_*)\in\Phi''$ with $T_*>0$,  assumptions $(H1)$ and $(H2)$,  and  there exist multipliers $(\mu, \varphi, e)$   satisfying  conditions $(i)$-$(iv)$ of Proposition \ref{Pro-KKT-P''}. If there exist positive constants $\gamma > 0$ and $\beta  > 0$ such that 
   \begin{align} 
   \label{StrictSOC}  
    2 \int_{Q_1} T\varphi \Big(-\Delta \zeta +    \alpha \zeta - v\Big)dxds  +  \mu \int_\Omega \zeta^2(1) dx \ge  \gamma (\|v\|^2_{L^2(Q_1)}  + T^2) \quad \forall (T, \zeta, v) \in {\mathcal C}_2^\beta[(T_*, \zeta_*, v_*)], 
\end{align}
then  there exist numbers $\varrho > 0$ and $\kappa > 0$ such that 
    \begin{align}
    \label{m00}
        T \ge T_* + \dfrac{\kappa}{2} (\|v - v_*\|_{L^2(Q_1)}^2 + (T - T_*)^2)
    \end{align}
    for all $z = (T, \zeta, v) \in \Phi''$ satisfying $|T  -  T_*| < \varrho$. In particular, $z_*$ is a locally strongly optimal solution to the problem  $(P'')$.
\end{theorem}

In the present setting, the Lagrange function has the following form
\begin{align}
\label{dfL}
    \mathcal L (T, \zeta, v, \varphi, \mu, e) = T  +  \int_{Q_1}\Big[\zeta(-\frac{\partial \varphi}{\partial s} - T \Delta \varphi) + (\alpha \zeta - v)T\varphi  + ev \Big]dxdt  +  \langle \varphi(1), \zeta(1)  \rangle  + \mu H(\zeta(1)).  
\end{align}
With the Lagrange function defined in (\ref{dfL}),  we can show that conditions $(i)$-$(iv)$ of Proposition \ref{Pro-KKT-P''} imply  $D_z\mathcal L (z_*, \varphi, \mu, e) =   0$. Also, the condition (\ref{StrictSOC}) can be rewritten as 
\begin{align}
  \label{StrictSOC.1}  
    D^2_{zz}\mathcal L (z_*, \varphi, \mu, e)[d, d] \ge  \gamma (\|v\|^2_{L^2(Q_1)} + T^2),   \quad \forall d =  (T, \zeta, v) \in {\mathcal C}_2^\beta[(T_*, \zeta_*, v_*)]. 
\end{align}
Define $\mathcal S[z_*] := \{(T, \zeta, v) \in \mathbb R \times Y_1 \times U_1: (T, \zeta, v) \ \  {\rm satisfies}\ \  (c_2)\}$. We need the following lemma.
\begin{lemma}
\label{B1}
    The mapping $\widehat G : \mathcal S[z_*] \to \mathbb R \times U_1$, $(T, \zeta, v) \mapsto \widehat G(T, \zeta, v) = (T, v)$, is bijective.
\end{lemma}
\begin{proof}
    For each $(T, v) \in \mathbb R \times U_1$, apply Lemma \ref{Lemma-stateEq}, the following equation 
    \begin{align}
        \frac{\partial \zeta}{\partial s}  - T_*\Delta \zeta + \alpha T_* \zeta = T_* v + T(\Delta \zeta_* - \alpha \zeta_* + v_*), \quad \zeta(0) = 0
    \end{align}
    has a unique solution $\zeta = \zeta(T, v) \in Y_1$. This implies that there exists unique $z = (T, \zeta, v) = (T, \zeta(T, v), v) \in \mathcal S[z_*]$ such that $(T, v) = \widehat G(z)$. Thus $\widehat G$ is bijective. 
\end{proof}

\noindent{\bf Proof of Theorem \ref{SOSC-P''}}.    Let $\varepsilon_0 \in (0, T_*)$.  Take $\varrho \in (0, \varepsilon_0]$ and $z = (T, \zeta, v) \in \Phi''$ such that $|T  -  T_*| < \varrho$.  Since $(T - T_*, v  - v_*) \in \mathbb R \times U_1$, by Lemma \ref{B1}, there exists $w_z  =  (T_w, \zeta_w, v_w) = (T_w, \zeta(T_w, v_w), v_w) \in \mathcal S[z_*]$ such that 
\begin{align}
    T_w =  T  - T_* \quad {\rm and} \quad v_w = v - v_*. 
\end{align}
Define $z_r = w_z  - (z  - z_*)$, that is
\begin{align}
z_r  = (T_r, \zeta_r, v_r) := \Big(T_w -  (T - T_*), \zeta_w  -  (\zeta - \zeta_*), v_w  - (v - v_*)\Big) = (0, \zeta_w  -  (\zeta - \zeta_*), 0).
\end{align}
Since $z, z_* \in \Phi''$, $F(z) = F(z_*) = 0$. We now apply the Taylor expansion for the functional $F$ at the point  $z_*$, there is $\widehat z = z_* + \theta(z - z_*)$, $\theta \in [0, 1]$ such that 
\begin{align}
    0 = F(z) - F(z_*) = D_zF(z_*)(z - z_*)  + \frac{1}{2}D^2_{zz}F(\widehat z)(z - z_*)^2, 
\end{align}
which implies that
\begin{align}
    - D_zF(z_*)(z - z_*)  =  \frac{1}{2}D^2_{zz}F(\widehat z)(z - z_*)^2. 
\end{align}
Since $w_z \in \mathcal S$, $D_zF(z_*)(w_z)  =  0$. Hence, we have 
\begin{align}
    D_zF(z_*)(z_r) = - D_zF(z_*)(w_z  - z_r)  = - D_zF(z_*)(z - z_*)  =  \frac{1}{2}D^2_{zz}F(\widehat z)(z - z_*)^2, 
\end{align}
which implies that 
\begin{align}
    \dfrac{\partial \zeta_r}{\partial s} - T_*\Delta \zeta_r +   \alpha T_* \zeta_r  = (T - T_*)\Big(-\Delta (\zeta - \zeta_*) + \alpha(\zeta - \zeta_*) - (v - v_*)\Big)  , \quad    \zeta_r(0)=0. 
\end{align}
 Subtracting the equations satisfied by $\zeta_* = \zeta_*(T_*, v_*)$ and $\zeta = \zeta(T, v)$ we have 
 \begin{align}
     \label{m0.1}
     \frac{\partial (\zeta_* - \zeta)}{\partial s} -  T_*\Delta (\zeta_* - \zeta) + \alpha T_*(\zeta_* - \zeta)  = T_*(v_*  - v) + (T_*  - T)[v + \Delta \zeta - \zeta], \quad  (\zeta_* - \zeta)(0)  = 0.
 \end{align}
 By regularity from \cite[Theorem 5, p. 360]{Evan}, there exist constants $c_1, c_2, c_3, c_4, c_5, c_6, c_7 > 0$ such that
\begin{align}
    \|v + \Delta \zeta - \zeta\|_{L^2(Q_1)}  &=  \frac{1}{T}  \|\frac{\partial \zeta}{\partial s}\|_{L^2(Q_1)}  \le \frac{1}{T_* -  \varepsilon_0}  \|\frac{\partial \zeta}{\partial s}\|_{L^2(Q_1)}  \nonumber \\    
    &\le \frac{1}{T_* -  \varepsilon_0}c_1(T\|v\|_{L^\infty(Q_1)} + \|y_0\|_{H^1_0(\Omega)})   \nonumber \\
    &\le \frac{1}{T_* -  \varepsilon_0} c_1\Big((T_*  + \varepsilon_0){\rm max}(|a|, b) + \|y_0\|_{H^1_0(\Omega)}\Big)  =: c_2,   \label{m0.2}  
\end{align}
\begin{align}
    \|\zeta   -  \zeta_*\|_{C([0, 1], H^1_0(\Omega))}  \le c_3 \|T_*(v_*  - v) + (T_*  - T)[v + \Delta \zeta - \zeta]\|_{L^2(Q_1)}   \le c_4(\|v - v_*\|_{L^2(Q_1)} + |T  -  T_*|)   \label{m0.3}
\end{align}
and
\begin{align}
    \|-\Delta (\zeta - \zeta_*) + \alpha(\zeta - \zeta_*) &- (v - v_*)\|_{L^2(Q_1)} \nonumber \\ 
    &\le    \frac{1}{T_*}\Big(\|\frac{\partial (\zeta_* - \zeta)}{\partial s}\|_{L^2(Q_1)}  +  \|(T_*  - T)[v + \Delta \zeta - \zeta]\|_{L^2(Q_1)}\Big)  \nonumber \\ &\le  \frac{1}{T_*}\Big(c_3 \|T_*(v_*  - v) + (T_*  - T)[v + \Delta \zeta - \zeta]\|_{L^2(Q_1)} + c_2|T  - T_*|\Big) \nonumber \\  
    &\le c_5(\|v - v_*\|_{L^2(Q_1)} + |T  -  T_*|),    \label{m0.4}
\end{align}
\begin{align}
    \|\zeta_r\|_{C([0, 1], H^1_0(\Omega))}  +   \|\frac{\partial \zeta_r}{\partial s}\|_{L^2(Q_1)} &\le c_6|T - T_*|\|-\Delta (\zeta - \zeta_*) + \alpha(\zeta - \zeta_*) - (v - v_*)\|_{L^2(Q_1)}  \nonumber \\
    &\le c_6(\|v  - v_*\|_{L^2(Q_1)}  +  |T - T_*|)\|-\Delta (\zeta - \zeta_*) + \alpha(\zeta - \zeta_*) - (v - v_*)\|_{L^2(Q_1)}  \nonumber \\
    &\le c_5c_6(\|v  - v_*\|_{L^2(Q_1)}  +  |T - T_*|)^2  \le 2 c_5c_6(\|v  - v_*\|^2_{L^2(Q_1)}  +  |T - T_*|^2),   \label{m0.5}
\end{align}
\begin{align}
    \|\zeta_r\|_{C([0, 1], H^1_0(\Omega))}      &\le c_6|T - T_*|\|-\Delta (\zeta - \zeta_*) + \alpha(\zeta - \zeta_*) - (v - v_*)\|_{L^2(Q_1)}  \nonumber \\
    &\le c_6|T - T_*|c_5(\|v - v_*\|_{L^2(Q_1)} + |T  -  T_*|)  \nonumber \\ 
    &\le c_5c_6 \varrho(\|v - v_*\|_{L^2(Q_1)} + |T  -  T_*|)   
    \label{m0.51}
\end{align}
and
\begin{align}
    &\|-\Delta \zeta_r +  \alpha \zeta_r\|_{L^2(Q_1)}  \le \frac{1}{T_*}\Big(\|\frac{\partial \zeta_r}{\partial s}\|_{L^2(Q_1)}  +   |T 
 - T_*|\|-\Delta (\zeta - \zeta_*) + \alpha(\zeta - \zeta_*) - (v - v_*)\|_{L^2(Q_1)}\Big) \nonumber \\
 &\le \frac{1}{T_*}\Big(\|\frac{\partial \zeta_r}{\partial s}\|_{L^2(Q_1)}  +   (\|v  - v_*\|_{L^2(Q_1)}  +  |T 
 - T_*|)\|-\Delta (\zeta - \zeta_*) + \alpha(\zeta - \zeta_*) - (v - v_*)\|_{L^2(Q_1)}\Big) \nonumber \\
 &\le \frac{1}{T_*}\Big(2 c_5c_6(\|v  - v_*\|^2_{L^2(Q_1)}  +  |T - T_*|^2) +  2 c_5(\|v  - v_*\|^2_{L^2(Q_1)}  +  |T - T_*|^2)\Big) \nonumber \\
 &=: c_7 (\|v  - v_*\|^2_{L^2(Q_1)}  +  |T - T_*|^2).    \label{m0.6}
\end{align}
Let us choose $\varrho \in (0, \varepsilon_0]$ such that
\begin{align}
    \varrho\Big(4c_7|Q_1|^{\frac{1}{2}}\|\varphi\|_{L^\infty(Q_1)}  +  2 \mu  c_5c_6(c_5c_6\varrho  + 2c_4)\Big)  \le \frac{\gamma}{2}.  \label{m0.8}
\end{align}
With $\varrho$ as above, take any $z = (T, \zeta, v) \in \Phi''$ such that $|T  -  T_*| < \varrho$, we consider the following cases.

\noindent {\bf{Case 1: $w_z \in \mathcal C _2^\beta[z_*]$.}} From facts $F(z) = F(z_*) = 0$, $\mu > 0$ and  $e \in N(K_1, v_*)$, we deduce that
\begin{align}
    \label{m1}
    T - T_* \ge \mathcal L (z, \varphi, \mu, e)  - \mathcal L (z_*, \varphi, \mu, e). 
\end{align}
By a Taylor expansion of the Lagrange function $z \mapsto \mathcal L (z, \varphi, \mu, e)$ at $z_*$, and recall that $D_z\mathcal L (z_*, \varphi, \mu, e) = 0$, there exist some $\widetilde z = z_* + \widehat \theta(z - z_*)$ with measurable function $0 \le \widehat \theta(x, s) \le 1$ such that 
\begin{align}
    \label{m2}
    T - T_* \ge \frac{1}{2} D^2_{zz}\mathcal L (\widetilde z, \varphi, \mu, e)(z - z_*)^2.
\end{align}
By the inequality $a \ge b - |a - b|$ for all real numbers $a, b$, we have 
\begin{align}
    \label{m3}
    &D^2_{zz}\mathcal L (\widetilde z, \varphi, \mu, e)(z - z_*)^2  \nonumber \\
    &\ge D^2_{zz}\mathcal L (z_*, \varphi, \mu, e)(z - z_*)^2  - |D^2_{zz}\mathcal L (\widetilde z, \varphi, \mu, e)(z - z_*)^2  -  D^2_{zz}\mathcal L (z_*, \varphi, \mu, e)(z - z_*)^2| \nonumber \\
    &= D^2_{zz}\mathcal L (z_*, \varphi, \mu, e)(z - z_*)^2  -  0 =  D^2_{zz}\mathcal L (z_*, \varphi, \mu, e)(z - z_*)^2. 
\end{align}
From (\ref{m2}) and (\ref{m3}), we get 
\begin{align}
    \label{m4}
    T - T_* \ge \frac{1}{2} D^2_{zz}\mathcal L (z_*, \varphi, \mu, e)(z - z_*)^2.
\end{align}
Since $z - z_* = w_z  - z_r$ and the matrix $D^2_{zz}\mathcal L (z_*, \varphi, \mu, e)$ is symmetric, 
\begin{align}
    \label{m5}
    D^2_{zz}\mathcal L (z_*, \varphi, \mu, e)(z - z_*)^2  &=  D^2_{zz}\mathcal L (z_*, \varphi, \mu, e)(w_z  - z_r)^2  \nonumber \\
    &= D^2_{zz}\mathcal L (z_*, \varphi, \mu, e)w_z^2  - \Big[2 D^2_{zz}\mathcal L (z_*, \varphi, \mu, e)(w_z, z_r)  -  D^2_{zz}\mathcal L (z_*, \varphi, \mu, e)z_r^2\Big]  \nonumber \\
    &= D^2_{zz}\mathcal L (z_*, \varphi, \mu, e)w_z^2  - D^2_{zz}\mathcal L (z_*, \varphi, \mu, e)(2w_z  -  z_r, z_r)  \nonumber \\
    &\ge D^2_{zz}\mathcal L (z_*, \varphi, \mu, e)w_z^2  - \Big|D^2_{zz}\mathcal L (z_*, \varphi, \mu, e)(2w_z  -  z_r, z_r)\Big|. 
\end{align}
Since $w_z \in \mathcal C _2^\beta[z_*]$, we have  from (\ref{StrictSOC.1}) that 
\begin{align}
    \label{m6}
    D^2_{zz}\mathcal L (z_*, \varphi, \mu, e)w_z^2 \ge \gamma (\|v_w\|^2_{L^2(Q_1)} + T_w^2)  = \gamma (\|v - v_*\|^2_{L^2(Q_1)}  +  |T - T_*|^2). 
\end{align}
Also, $(2w_z  -  z_r, z_r)  = \Big(2T_w, \zeta_r + 2(\zeta - \zeta_*) , 2v_w\Big)$ and $z_r = (0, \zeta_r, 0)$. Hence we have
\begin{align}
    \label{m7}
    &\Big|D^2_{zz}\mathcal L (z_*, \varphi, \mu, e)(2w_z  -  z_r, z_r)\Big|  \nonumber \\
    &= \Big|4 \int_{Q_1} \varphi T_w (-\Delta \zeta_r + \alpha \zeta_r) dxds  + \mu \int_\Omega \Big(\zeta_r(1) + 2[\zeta(1) - \zeta_*(1)]\Big)\zeta_r(1) dx \Big|  \nonumber \\
    &\le 4 \int_{Q_1} \Big|\varphi T_w (-\Delta \zeta_r + \alpha \zeta_r)\Big| dxds  + \mu \int_\Omega \Big|\Big(\zeta_r(1) + 2[\zeta(1) - \zeta_*(1)]\Big)\zeta_r(1)\Big|dx  \nonumber \\
    &\le 4\varrho|Q_1|^{\frac{1}{2}}\|\varphi\|_{L^\infty(Q_1)}\|-\Delta \zeta_r + \alpha \zeta_r\|_{L^2(Q_1)}  + \mu \Big(\|\zeta_r(1)\|_{L^2(\Omega)} + 2\|\zeta(1) - \zeta_*(1)\|_{L^2(\Omega)}\Big)\|\zeta_r(1)\|_{L^2(\Omega)}  \nonumber \\
    &\le 4\varrho|Q_1|^{\frac{1}{2}}\|\varphi\|_{L^\infty(Q_1)}\|-\Delta \zeta_r + \alpha \zeta_r\|_{L^2(Q_1)}  \nonumber \\
    &+ \mu \Big(\|\zeta_r\|_{C([0, 1], H^1_0(\Omega))} + 2\|\zeta - \zeta_*\|_{C([0, 1], H^1_0(\Omega))}\Big)\|\zeta_r\|_{C([0, 1], H^1_0(\Omega))}. 
\end{align}
From (\ref{m7}), (\ref{m0.6}), (\ref{m0.51}), (\ref{m0.5}) and (\ref{m0.3}) we have 
\begin{align}
    \label{m8}
    &\Big|D^2_{zz}\mathcal L (z_*, \varphi, \mu, e)(2w_z  -  z_r, z_r)\Big|  \nonumber \\
    &\le 4c_7\varrho|Q_1|^{\frac{1}{2}}\|\varphi\|_{L^\infty(Q_1)} (\|v  - v_*\|^2_{L^2(Q_1)}  +  |T - T_*|^2)   \nonumber \\
    &+  \mu \Big(c_5c_6\varrho(\|v - v_*\|_{L^2(Q_1)} + |T  -  T_*|)  +  2 c_4(\|v - v_*\|_{L^2(Q_1)} + |T  -  T_*|)\Big)c_5c_6\varrho(\|v - v_*\|_{L^2(Q_1)} + |T  -  T_*|)  \nonumber \\
    &\le \varrho\Big(4c_7|Q_1|^{\frac{1}{2}}\|\varphi\|_{L^\infty(Q_1)}  +  2 \mu  c_5c_6(c_5c_6\varrho  + 2c_4)\Big)(\|v  - v_*\|^2_{L^2(Q_1)}  +  |T - T_*|^2) 
\end{align}
Combining this with (\ref{m0.8}), we have 
\begin{align}
    \label{m9}
    \Big|D^2_{zz}\mathcal L (z_*, \varphi, \mu, e)(2w_z  -  z_r, z_r)\Big|   \le \frac{\gamma}{2}  (\|v  - v_*\|^2_{L^2(Q_1)}  +  |T - T_*|^2).  
\end{align}
From (\ref{m5}), (\ref{m6}) and (\ref{m9}) we get 
\begin{align}
    \label{m10}
    D^2_{zz}\mathcal L (z_*, \varphi, \mu, e)(z - z_*)^2   &\ge (\gamma - \frac{\gamma}{2})  (\|v  - v_*\|^2_{L^2(Q_1)}  +  |T - T_*|^2)  \nonumber \\
    &= \frac{\gamma}{2}  (\|v  - v_*\|^2_{L^2(Q_1)}  +  |T - T_*|^2).
\end{align}
It follows from (\ref{m4}) and (\ref{m10}) that 
\begin{align}
    \label{m11}
    T  - T_* \ge  \frac{\gamma}{4}  (\|v  - v_*\|^2_{L^2(Q_1)}  +  |T - T_*|^2),
\end{align}
which implies that the condition (\ref{m00}) holds true with $\kappa \in (0,  \frac{\gamma}{2}]$.

\noindent {\bf{Case 2: $w_z \notin \mathcal C _2^\beta[z_*]$.}} We have $w_z \in \mathcal S[z_*]$ and it's easy to see that $v_w = v  - v_* \ge 0$ a.e. on $Q_a$ and $v_w = v  - v_* \le 0$ a.e. on $Q_b$. By the definition of the cone $\mathcal C _2^\beta[z_*]$, we deduce that 
\begin{align}
    T - T_* = T_w &> \beta (\|v_w\|_{L^2(Q_1)}  +  |T_w|)  \nonumber \\  
    &= \beta (\|v - v_*\|_{L^2(Q_1)}  +  |T - T_*|).  \nonumber \\  
    &\ge \beta \Big(\frac{1}{(b - a)|Q_1|^{\frac{1}{2}}}\|v - v_*\|^2_{L^2(Q_1)}  +  \frac{1}{\varepsilon_0}|T - T_*|^2\Big).  \nonumber \\ 
    &\ge \beta {\rm min} \Big(\frac{1}{(b - a)|Q_1|^{\frac{1}{2}}}, \frac{1}{\varepsilon_0}\Big) (\|v - v_*\|^2_{L^2(Q_1)}  +  |T - T_*|^2),  
\end{align}
which implies that $z_*$ is a locally strongly optimal solution
to the problem, and in this case the condition (\ref{m00}) holds true for all $\kappa$ such that $0 < \kappa \le \beta {\rm min} \Big(\frac{1}{(b - a)|Q_1|^{\frac{1}{2}}}, \frac{1}{\varepsilon_0}\Big)$.

\noindent The proof of the theorem is complete.
$\hfill\square$

\section{ Numerical approximation for $(P'')$}

In this section we study the numerical discretization of $(P'')$ by discontinuous Galerkin finite element methods. For this, we will further assume that $\Omega$  is convex. We will discretize both the state and the control using functions that are continuous piecewise linear in space and piecewise constant in time. To this aim,  associated with a parameter $h$,  we will consider, cf. Brenner \& Scott  \cite[Definition (4.4.13)]{0.7},   a quasi-uniform family of triangulations $\{\mathcal {K}_{h}\}_{h>0}$ of $\bar \Omega$ and a quasi-uniform family of partitions of size $\tau$ of $[0, 1]$,  that is
$$0=t_0<t_1<\cdots <t_{N_\tau }=  1.$$
We will denote by  
$N_h$ and $N_{I, h}$ the number of nodes and interior nodes of  ${\mathcal {K}}_{h}, I_j=(t_{j-1},t_j), \tau _j = t_j-t_{j-1}$, $\tau = \max \{\tau _j:\,1\le j\le N_\tau \}$, the discretization parameter $h$ as the cellwise constant function $h|_K = h_K$ with diameter $h_K$ of the cell $K$ and set $h = {\rm max}h_K$, 
$$\sigma := (\tau, h).$$ 
Following Casas,  Kunisch and Mateos \cite[Assumption 4]{Casas-2023-2},  we need the following assumption. 

\noindent $(H4)$ \ There exist positive constants $\theta_1, \theta_2, \theta_3, \theta_4, \theta_5, c_\Omega, C_\Omega > 0$ which are independent of $\tau$ and $h$ such that
\begin{align}
    \tau_j \ge \theta_1 \tau^{\theta_2}, \ \tau \le \theta_3\tau_j \ {\rm {for \ all}} \ j = 1, 2, ..., N_{\tau} \quad {\rm and} \quad c_{\Omega} h^{\theta_4} \le \tau \le C_{\Omega} h^{\theta_5}. 
\end{align}

\subsection{Approximation of the state equation}

Now we consider the finite dimensional spaces
\begin{align}
Y_h = \{ {\zeta}_h\in C(\bar{\Omega }):\ {\zeta}_{h|K}\in P_1(K)\ \forall K\in \mathcal {K}_h,\ \zeta_h\equiv 0 \text{ in } \bar{\Omega }{\setminus }\Omega _h\},
\end{align}
where $P_1(K)$ is the space of polynomials of degree less or equal than $1$ on the element $K$, and 
\begin{align}
 \mathcal {Y}_\sigma =\{\zeta_\sigma \in L^2(0, 1;Y_h): \zeta_{\sigma |I_j}\in Y_h\ \forall j =1,\ldots ,N_{\tau }\}. 
\end{align}
Then each element $\zeta_\sigma$ of  $\mathcal {Y}_\sigma $  can be written as
\begin{align}
\zeta_\sigma = \sum _{j=1}^{N_\tau }\zeta_{ {h,j}}\chi _j = \sum _{j=1}^{N_\tau }\sum _{i=1}^{N_{I,h}}\zeta_{i,j} e_i \chi _j,  
\end{align}
where  $\zeta_{h, j} \in Y_h$,  for $j =  1, 2, ..., N_\tau$, $\zeta_{i, j} \in \mathbb R$ for $i =  1, 2, ..., N_{I, h}$ and $j =  1, 2, ..., N_\tau$,   $\{e_i\}_{i = 1}^{N_{ {I,h}}}$  is the nodal basis associated to the interior nodes $\{x_i\}_{i = 1}^{N_{ {I,h}}}$ of the triangulation and  $\chi _j$ denotes the characteristic function of the interval  $I_j = (t_{j-1},t_j)$.   For every $T > 0$ and $v\in {L^2(Q_1)}$, we define its associated discrete state as the unique element $\zeta_\sigma =  \zeta_\sigma(T, v) \in \mathcal {Y}_\sigma$ such that for $j = 1,\ldots ,N_\tau$, 
\begin{align}
&\int _{\Omega} (\zeta_{h,j}-  \zeta_{h,j-1}) z_h dx   +   {\tau _j} T a( \zeta_{h,j},z_h)  +  \int _{I_j}\int _{\Omega}  \alpha T \zeta_{h,j} z_h dx ds  = \int _{I_j}\int _{\Omega} T v z_h dx ds, \quad    \    \forall z_h\in Y_h    \label{ds1}\\
&\int _{\Omega} \zeta_{h,0} z_h dx = \int _{\Omega} y_{0} z_h dx,  \quad   \ \forall z_h\in Y_h  \label{ds2}
\end{align}
  where 
\begin{align}
 a(\zeta, z) =  \int _{\Omega}\sum _{i,j=1}^N \partial _{x_i} \zeta \partial _{x_j} zdx,   \quad  {\forall \zeta, z \in H^1(\Omega)}.  
\end{align}
This scheme can be interpreted as an implicit Euler discretization of the system of ordinary differential equations obtained after spatial finite element discretization. The proof of existence and uniqueness of a solution of (\ref{ds1})-(\ref{ds2}) is standard by using Brouwer’s fixed point theorem.  Moreover, by a similar way as the work of E. Casas, K. Kunisch and M. Mateos \cite[Theorem 3.1]{Casas-2023-2}, we deduce that  the system (\ref{ds1})-(\ref{ds2}) realizes an approximation of the state equation (\ref{P2.1})-(\ref{P3.1}) in the following sense.
\begin{lemma}
\label{Lemma-stateEq-d}
Suppose that $(H1)$ and $(H4)$ are valid. Then for each $T>0$ and  $v\in L^p(Q_1)$ with $p$ satisfying \eqref{Dimension1}, the discrete state equation (\ref{ds1})-(\ref{ds2}) has a unique  solution    $\zeta_\sigma = \zeta_\sigma(T, v) \in {\mathcal Y}_\sigma$ and there exist positive constant $h_0, \tau_0, \delta_0, c_1, c_2 > 0$ which are independent of $T$ and $v$ such that for every $\tau < \tau_0$ and $h < h_0$
\begin{align}
    &\|\zeta(T, v)  -   \zeta_{\sigma}(T, v)\|_{L^2(Q_1)}   \le c_1 \Big(T\|v\|_{L^p(Q_1)} + C(y_0)\Big)(\tau + h^2), \label{KeyInq1} \\
    &\|\zeta(T, v)  -   \zeta_{\sigma}(T, v)\|_{L^\infty(Q_1)}   \le c_2 \Big(T\|v\|_{L^p(Q_1)} + C(y_0)\Big)|{\rm log}h|^3 h^{\delta_0},  \label{KeyInq2} 
\end{align} 
where $C(y_0) := \|y_0\|_{C^{0, \alpha}(\bar \Omega)} + \|y_0\|_{H_0^1(\Omega)}$. 
\end{lemma}

\subsection{Discretization version  of the optimal control problem} 

To discretize the controls, we will also use continuous piecewise linear in space and piecewise constant in time functions.  We define
\begin{align}
 V_h= & {} \{v_h\in C(\bar\Omega ):\ v_{h|K}\in P_1(K)\ \ \forall K\in \mathcal {K}_h\},  \\
 \mathcal {V}_\sigma= & {} \{v_\sigma \in L^2(0, 1;  V_h): v_{\sigma |I_j}\in V_h\ \forall j =1,\ldots ,N_{\tau }\}. 
\end{align}
Then, the elements of $\mathcal {V}_\sigma$  can be written as
\begin{align}
 v_\sigma = \sum _{j=1}^{N_\tau }v_{h,j}\chi _j = \sum _{j= 1}^{N_\tau }\sum _{i=1}^{N_{h}}v_{i,j} e_i \chi _j= {\sum _{i=1}^{N_h} v_{{\tau ,i}} e_i}, 
\end{align}
where $v_{h,j}\in V_h$ for $j=1,\ldots ,N_\tau,$ $v_{i,j}\in \mathbb {R}$ for $i=1,\ldots ,N_{h}$ and $j =1,\ldots ,N_\tau$;  $\{e_i\}_{i = 1}^{N_h}$ is the nodal basis associated to all nodes $\{x_i\}_{i = 1}^{N_h}$  of the triangulation. Notice that  $v_{{\tau ,i}}=\sum _{j=1}^{N_\tau }v_{i,j}\chi _j$ is a piecewise constant function of time.

Error estimates for solutions to continuous and discrete state equations are the following. 

\begin{lemma}
\label{Lemma-stateEq-dd}
Suppose that $(H1)$ and $(H4)$ are valid. Then for $T_\sigma, T>0$ and  $v_\sigma \in \mathcal{V}_\sigma$,  $v\in L^p(Q_1)$ with $p$ satisfying \eqref{Dimension1},  there exist positive constant $h_0, \tau_0, \delta_0, C_1, C_2 > 0$ which are independent of  $T_\sigma, T$  and  $v_\sigma, v$  such that for every $\tau < \tau_0$ and $h < h_0$, 
\begin{align}
    &\|\zeta_{\sigma}(T_\sigma, v_\sigma)  -   \zeta(T, v)\|_{L^2(Q_1)}   \le C_1 \Big[\Big(T\|v_{\sigma}\|_{L^p(Q_1)} + C(y_0)\Big)(\tau + h^2)   + C_2(v)|T_\sigma  - T|  + T_\sigma \|v_\sigma - v\|_{L^2(Q_1)} \Big], \label{KeyInq11} \\
    &\|\zeta_{\sigma}(T_\sigma, v_\sigma)  -   \zeta(T, v)\|_{L^\infty(Q_1)}   \le C_2 \Big[\Big(T\|v_{\sigma}\|_{L^p(Q_1)} + C(y_0)\Big)|{\rm log}h|^3 h^{\delta_0}    + C_p(v)|T_\sigma  - T|   + T_\sigma \|v_\sigma - v\|_{L^p(Q_1)} \Big], \label{KeyInq22} 
\end{align} 
where $C(y_0) := \|y_0\|_{C^{0, \alpha}(\bar \Omega)} + \|y_0\|_{H_0^1(\Omega)}$,  $C_2(v) := \|v - A\zeta - \psi(\zeta)\|_{L^2(Q_1)}$ and $C_p(v) := \|v - A\zeta - \psi(\zeta)\|_{L^p(Q_1)}$ with $\zeta = \zeta(T, v)$.  Particularly for $s = 1$, one has 
\begin{align}
    \|\zeta_{\sigma}(\cdot, 1)  -   \zeta(\cdot, 1)\|_{L^\infty(\Omega)}   \le C_2 \Big[\Big(T\|v_{\sigma}\|_{L^p(Q_1)} + C(y_0)\Big)|{\rm log}h|^3 h^{\delta_0}    + C_p(v)|T_\sigma  - T|   + T_\sigma \|v_\sigma - v\|_{L^p(Q_1)} \Big]. \label{KeyInq23} 
\end{align}
\end{lemma}
\begin{proof}
    Notice that 
    \begin{align*}
        |\zeta_{\sigma}(T_\sigma, v_\sigma)  -   \zeta(T, v)| \le  |\zeta_{\sigma}(T_\sigma, v_\sigma)  - \zeta(T_\sigma, v_\sigma)| +  |\zeta(T_\sigma, v_\sigma)   -   \zeta(T, v)|, 
    \end{align*}
    and $\zeta(T_\sigma, v_\sigma)   -   \zeta(T, v)$ is solution to the following parabolic equation 
    \begin{align}
    \begin{cases}
        \dfrac{\partial(\zeta_\sigma - \zeta)}{\partial s}  + T_\sigma A(\zeta_\sigma - \zeta) + T_\sigma[\psi(\zeta_\sigma) - \psi(\zeta)] =       T_\sigma(v_\sigma - v) + (T_\sigma - T)[v - A\zeta - \psi(\zeta)],\\
        (\zeta_\sigma - \zeta)(0)  = 0. 
    \end{cases}
    \end{align}
    Then, estimates (\ref{KeyInq11}) and (\ref{KeyInq22}) follow from Lemma \ref{Lemma-stateEq}, \cite[Theorem 5, p. 360]{Evan} and Lemma \ref{Lemma-stateEq-d}. 
\end{proof}

\medskip

Now, we formulate the discrete control problem as follows 
\begin{align}
    &J(T_\sigma,  \zeta_\sigma, v_\sigma) =  T_\sigma  \to  \inf    \label{P1.1d}\\
    &{\rm s.t.}  \quad
    (T_\sigma,  \zeta_\sigma, v_\sigma) \in Z_\sigma := (0, +\infty)\times {\mathcal{Y}}_\sigma \times {\mathcal V}_\sigma,  
    \  \zeta_\sigma = \sum _{j=1}^{N_\tau }\zeta_{ {h,j}}\chi _j, \   v_\sigma = \sum _{j=1}^{N_\tau }v_{h,j}\chi _j,   \nonumber \\ 
    (P''_\sigma)  \quad  \quad \quad  &{\rm {for}} \  j = 1,\ldots ,N_\tau,  \quad  \int _{\Omega} (\zeta_{h,j}-  \zeta_{h,j-1}) z_h dx  +   {\tau _j} T_\sigma a( \zeta_{h,j},z_h)    \nonumber \\ 
    &\hspace{2.7cm} +  \int _{I_j}\int _{\Omega} \alpha T_\sigma \zeta_{h,j} z_h dx ds  = \int _{I_j}\int _{\Omega}T_\sigma v_{h,j} z_h dx ds \quad \forall z_h\in Y_h,    \label{P2.1d}\\
    &\int _{\Omega} \zeta_{h,0} z_h dx = \int _{\Omega} y_{0} z_h dx  \quad    \forall z_h\in Y_h,   \label{P3.1d}\\
    &H(\zeta_\sigma(1))  \le 0,       \label{P4.1d}\\
    &a \le v_\sigma(x, s) \le  b \quad {\rm a.a.} \ x \in \Omega, \ s  \in [0, 1].   \label{P5.1d}
\end{align}
This is a mathematical programming problem in finite dimensional spaces. Thus the discrete version is a sequence of  approximate optimal control problems to the problem $(P'')$. Hereafter, we denote by $\Phi''_\sigma$ the feasible sets of $(P''_\sigma)$.  

\newpage

To end this section, we present optimality conditions for $(P''_\sigma)$.  

\begin{proposition}
\label{Pro-KKT-P''_sigma}
Suppose assumptions $(H1)$-$(H3)$, $T_{*\sigma} > 0$ and $z_{*\sigma} = (T_{*\sigma}, \zeta_{*\sigma}, v_{*\sigma})\in\Phi''_\sigma$ is a locally optimal solution to $(P''_\sigma)$, where $\zeta_{*\sigma} = \sum _{j=1}^{N_\tau }\zeta_{ {h,j}}\chi _j$  and    $v_{*\sigma} = \sum _{j=1}^{N_\tau}v_{h,j}\chi _j$. Then there exist Lagrange multipliers $\mu_\sigma > 0$ and  $\varphi_\sigma, e_\sigma \in {\mathcal Y}_\sigma$ with  $ \varphi_\sigma = \sum _{j=1}^{N_\tau }\varphi_{ {h,j}}\chi _j$  and    $e_\sigma = \sum _{j=1}^{N_\tau}e_{h,j}\chi _j$ such that the following conditions are satisfied:

\noindent $(i)$  (the discrete adjoint equations)  for $j  =  N_\tau,\ldots , 1$, 
\begin{align}
&\int _{\Omega} (\varphi_{h,j}-  \varphi_{h,j + 1}) z_h dx  +   {\tau _j} T_{*\sigma} a(z_h, \varphi_{h,j})     +  \int _{I_j}\int _{\Omega} \alpha T_{*\sigma} \varphi_{h,j} z_h dx ds  =   0    \quad \forall z_h\in Y_h, \label{cm0.1b}   \\
    &\int _{\Omega} \varphi_{h,  N_{\tau + 1}} z_h dx =   - \mu_\sigma \int _{\Omega} (\zeta_{h,  N_{\tau + 1}} -  y_\Omega) z_h dx  \quad    \forall z_h\in Y_h;   \label{cm0.1b1}
\end{align}

\noindent $(ii)$   (optimality condition for $v_{*\sigma}$) 
\begin{align}
\label{cm0.3b}
      -  T_{*\sigma}\varphi_\sigma (x, s) +  e_\sigma(x, s)  =  0 \quad {\rm a.a.} \ (x, s) \in Q_1; 
\end{align}

\noindent $(iii)$   (optimality condition for $T_{*\sigma}$) 
\begin{align}\label{cm0.4b}  
     \sum_{j = 1}^{N_\tau}\Big\{ {\tau _j} a( \zeta_{h,j}, \varphi_{h,j})   
     +  \int _{I_j}\int _{\Omega} (\alpha  \zeta_{h,j}  - v_{h,j}) \varphi_{h,j} dx ds \Big\} = -1;
\end{align}

\noindent $(iv)$ (the complementary conditions)   $H(\zeta_{h, N_{\tau + 1}}) = 0$ and
\begin{align}
\label{a0.4b}
    e_\sigma(x, s)
    \begin{cases}
       \leq 0\quad  &{\rm {if}} \ v_{h, j}(x) = a, \\
    \geq 0\quad  &{\rm {if}} \ v_{h, j}(x) = b, \\
     =0 \quad &{\rm {if}} \ v_{h, j}(x) \in (a, b).
    \end{cases}
\end{align}
\end{proposition}
\begin{proof}  
    The proof of this result can be done in the same way as the proof of optimality conditions in \cite[Lemma 4.9]{Bon1}. More precisely, by $(H3)$ and similar arguments as \cite[Proposition 4.7]{Bon1}, the linearized Slater condition holds at $z_{*\sigma} = (T_{*\sigma}, \zeta_{*\sigma}, v_{*\sigma})\in\Phi''_\sigma$, which  yields the optimality conditions in qualified form.  In our case, the Lagrange function associated with the problem $(P''_\sigma)$ is given by:
    \begin{align}
        &{\mathcal L}_\sigma(T_\sigma, \zeta_\sigma, v_\sigma; \varphi_\sigma, \mu_\sigma, e_\sigma)  =  T_\sigma \nonumber \\
        &+ \sum_{j = 1}^{N_\tau}\Big\{\int _{\Omega} (\zeta_{h,j}-  \zeta_{h,j-1}) \varphi_{h,j} dx  +   {\tau _j} T_\sigma a( \zeta_{h,j}, \varphi_{h,j})   
     +  \int _{I_j}\int _{\Omega} \alpha T_\sigma \zeta_{h,j} \varphi_{h,j} dx ds  - \int _{I_j}\int _{\Omega}T_\sigma v_{h,j} \varphi_{h,j} dx ds \Big\} \nonumber \\
     &+  \int _{\Omega} (\zeta_{h,0} - y_0)  \varphi_{h,0} dx    + \mu_\sigma H(\zeta_{h, N_{\tau + 1}})  + \sum_{j = 1}^{N_\tau}\int _{I_j}\int _{\Omega}e_{h,j} v_{h,j} dx ds.  
    \end{align}
    Then the optimality of $z_{*\sigma}$ implies that there exist Lagrange multipliers $\mu_\sigma \ge 0$ and  $\varphi_\sigma, e_\sigma \in {\mathcal Y}_\sigma$ with  $ \varphi_\sigma = \sum _{j=1}^{N_\tau }\varphi_{ {h,j}}\chi _j$  and    $e_\sigma = \sum _{j=1}^{N_\tau}e_{h,j}\chi _j$ such that for $j  =  N_\tau,\ldots , 1$, 
    \begin{align}
        &(\mu_\sigma, \varphi_\sigma, e_\sigma) \ne (0, 0, 0), \label{xx0} \\
        &\int _{\Omega} (\varphi_{h,j}-  \varphi_{h,j + 1}) z_h dx  +   {\tau _j} T_{*\sigma} a(z_h, \varphi_{h,j})     +  \int _{I_j}\int _{\Omega} \alpha T_{*\sigma} \varphi_{h,j} z_h dx ds  =   0    \quad \forall z_h\in Y_h,   \label{xx1}  \\
        &\int _{\Omega} \varphi_{h,  N_{\tau + 1}} z_h dx =   - \mu_\sigma \int _{\Omega} (\zeta_{h,  N_{\tau + 1}} -  y_\Omega) z_h dx  \quad    \forall z_h\in Y_h, \label{xx2} \\
        &\int _{I_j}\int _{\Omega}\big(-T_{*\sigma}  \varphi_{h,j} + e_{h,j}\big) z_h dx ds  = 0  \quad \forall z_h\in Y_h,    \label{xx3}  \\
        &1 + \sum_{j = 1}^{N_\tau}\Big\{ {\tau _j} a( \zeta_{h,j}, \varphi_{h,j})   
     +  \int _{I_j}\int _{\Omega} (\alpha  \zeta_{h,j}  - v_{h,j}) \varphi_{h,j} dx ds \Big\} = 0, \label{xx4} \\
        &\mu_\sigma H(\zeta_{h, N_{\tau + 1}})  = 0, \label{xx5} \\
        &e_\sigma(x, s)\in  N([a, b],  v_{*\sigma}(x, s)) \ {\rm a.a.}\ (x, s)\in Q_1. \label{xx6}
    \end{align}
    From (\ref{xx1}), (\ref{xx2}), (\ref{xx3}) and (\ref{xx4}), it is easy to see that $(i)$, $(ii)$ and $(iii)$ are valid. \\
If $\mu_\sigma = 0$ then we have from (\ref{xx1})-(\ref{xx2}) that $\varphi_{h, j} = 0$ for all $j = 1, 2, ..., N_\tau$, and so $\varphi_\sigma = 0$. This and (\ref{xx3}) imply that $e_\sigma = 0$,  which contradicts (\ref{xx0}). So we must have $\mu_\sigma > 0$.  \\
Combining (\ref{xx5}) with $\mu_\sigma > 0$, we get $H(\zeta_{h, N_{\tau + 1}}) = 0$. The condition (\ref{xx6}) implies (\ref{a0.4b}). 
The proof of the proposition is complete.    
\end{proof}

\subsection{Convergence analysis and error estimates}

For each $v \in L^1(Q_1)$, we define 
\begin{align}
 I_\sigma v := P_\tau E_h v = E_h P _\tau v = \sum _{i = 1}^{N_h}\sum _{j = 1}^{N_\tau }\Big (\frac{1}{\tau _j\int _{\Omega} e_i\, dx}\int _{t_{j-1}}^{t_j}\int _{\Omega} v(x, s)e_i(x)\, dx\, ds\Big )e_i\chi _j, 
\end{align}
where $P_\tau$ is the $L^2(0, 1)$ projection operator and $E_h : L^1(\Omega) \to V_h$ is the Carstensen quasi-interpolation operator (see \cite{CARSTENSEN}). Then $I_\sigma v \in \mathcal V_\sigma$. 
\begin{lemma}
\label{Proj}
    If $v \in L^2(Q_1)$ then $I_\sigma v  \to v$ in $L^2(Q_1)$ as $\sigma \to 0$. Moreover, if $v \in H^1(Q_1)$ then there exists $C > 0$ independent of $\tau$ and $h$ such that 
    \begin{align}
        \|I_\sigma v - v\|_{L^2(Q_1)}  \le C(\tau + h) \|v\|_{H^1(Q_1)}. 
    \end{align}
\end{lemma}
\begin{proof}
    The reader is referred to \cite[Lemma 6.3 and Lemma 6.6]{Casas-2018} for the detailed proof of the result.

\end{proof}

The following theorem is main result of this section. 
\begin{theorem}
\label{main-result}
    Suppose that $(H1)$-$(H4)$ are satisfied, $\Phi''_+$ is nonempty and $\{(\bar T_\sigma, \bar \zeta_\sigma, \bar v_\sigma)\}$ is a sequence of  globally optimal solutions to $(P''_\sigma)$. Then:
    \begin{enumerate}[$(i)$]
        \item There exists a subsequence $\{(\widehat T_\sigma, \widehat \zeta_\sigma, \widehat v_\sigma)\}$ of $\{(\bar T_\sigma, \bar \zeta_\sigma, \bar v_\sigma)\}$ such that $\widehat T_\sigma \to \widehat T$ in $\mathbb R$, $\widehat \zeta_\sigma \to \widehat \zeta$ in $L^\infty(Q_1)$ and $\widehat v_\sigma \rightharpoonup  \widehat v$ in $L^p(Q_1)$, where $(\widehat T, \widehat \zeta, \widehat v)$ is a globally optimal solution of the problem ${\rm inf}_{(T,  \zeta, v) \in \Phi''_+} T$ with $\Phi''_+$ defined by (\ref{dfphi+});
        \item In addition, if the strong second-order sufficient condition (\ref{StrictSOC}) is satisfied for some $\gamma, \beta > 0$ and $(\mu, \varphi, e)$   satisfying  conditions $(i)$-$(iv)$ of Proposition \ref{Pro-KKT-P''}, then the convergence property in part $(i)$ holds for the entire sequence $\{(\bar T_\sigma, \bar \zeta_\sigma, \bar v_\sigma)\}$, namely,      
        $\bar T_\sigma \to T_*$ in $\mathbb R$, $\bar \zeta_\sigma \to \zeta_*$ in $L^\infty(Q_1)$ and $\bar v_\sigma \to  v_*$ in $L^s(Q_1)$ for any $1 \le s < +\infty$. Moreover, there exist positive constants $h_0, \tau_0, \delta_0, c_i > 0$, $i = 1, ..., 8$,  which are independent of $\sigma$ such that 
        \begin{align}
        &|\bar T_\sigma - T_*| \le c_1 \big(\tau + h  +  |{\rm log} \tau|(\tau + h^2)\big),  \label{es1}\\
        &\|\bar v_\sigma -  v_*\|_{L^s(Q_1)} \le c_2 \big(\tau + h  +  |{\rm log} \tau|(\tau + h^2)\big)^{\frac{1}{s}} \quad \forall s \ge 2,  \label{es2} \\
            &\|\bar \zeta_\sigma - \zeta_*\|_{L^2(Q_1)}   \le c_3(\tau + h^2) + c_4 \big(\tau + h  +  |{\rm log} \tau|(\tau + h^2)\big) + c_5 \Big(\tau + h  +  |{\rm log} \tau|(\tau + h^2)\Big)^{\frac{1}{2}}, \label{es3}  \\
            &\|\bar \zeta_\sigma - \zeta_*\|_{L^\infty(Q_1)}   \le c_6|{\rm log}h|^3 h^{\delta_0} + c_7  \big(\tau + h  +  |{\rm log} \tau|(\tau + h^2)\big) + c_8 \Big(\tau + h  +  |{\rm log} \tau|(\tau + h^2)\Big)^{\frac{1}{p}} \label{es4} 
        \end{align}
        for every $\tau \le \tau_0$ and $h \le h_0$. 
    \end{enumerate}
\end{theorem}
\begin{proof}
    $(i)$ \  The assumptions $(H1), (H2)$ are valid and $\Phi''_+$ is nonempty, by  Proposition \ref{E3}, the problem ${\rm inf}_{(T,  \zeta, v) \in \Phi''_+} T$ has at least one globally optimal solution $(T_*, \zeta_*, v_*) \in \Phi''_+$. By similar arguments as \cite[Proposition 4.4]{Bon1}, there exist $\tau_1, h_1 > 0$ and a sequence $\{(\widetilde T_\sigma, \widetilde \zeta_\sigma, \widetilde v_\sigma)\}$ with $(\widetilde T_\sigma, \widetilde \zeta_\sigma, \widetilde v_\sigma) \in \Phi''_\sigma$ for $\tau \le \tau_1$, $h \le h_1$ and $\widetilde T_\sigma \to T_*$ as $\sigma \to 0$. This fact implies that there exists $m > 0$ such that $0 < \bar T_\sigma \le \widetilde T_\sigma \le m$ for $\tau, h$ small enough. Hence, there exists $\{\widehat T_\sigma\} \subseteq \{\bar T_\sigma\}$ such that $\widehat T_\sigma \to \widehat T \ge 0$. It's easy to see that $\{\bar v_\sigma\}$ is bounded in $L^p(Q_1)$.  By the reflexivity of $L^p(Q_1)$, we can extract a subsequence $\{\widehat v_\sigma\} \subseteq\{\bar v_\sigma\}$ such that $\widehat v_\sigma \rightharpoonup  \widehat v$ in $L^p(Q_1)$. By weak closeness in $L^p(Q_1)$ of the set $K_p$ defined as (\ref{Kp}), we have 
    \begin{align}
    a \le \widehat v(x, s) \le b  \ {\rm a.a.} \  (x, s) \in Q_1.
    \end{align}
    Also, $\zeta(\widehat T_\sigma, \widehat v_\sigma)   -   \zeta(\widehat T, \widehat v)$ is solution of 
     \begin{align}
        &\dfrac{\partial(\zeta(\widehat T_\sigma, \widehat v_\sigma)   -   \zeta(\widehat T, \widehat v))}{\partial s}  + \widehat T_\sigma A(\zeta(\widehat T_\sigma, \widehat v_\sigma)   -   \zeta(\widehat T, \widehat v)) \nonumber \\
        &+ \widehat T_\sigma\psi'(\zeta(\widehat T, \widehat v) + \theta (\zeta(\widehat T_\sigma, \widehat v_\sigma)   -   \zeta(\widehat T, \widehat v)))(\zeta(\widehat T_\sigma, \widehat v_\sigma)   -   \zeta(\widehat T, \widehat v))  \nonumber \\
        &=       \widehat T_\sigma(\widehat v_\sigma - \widehat v) + (\widehat T_\sigma - \widehat T)[\widehat v - A\zeta(\widehat T, \widehat v) - \psi(\zeta(\widehat T, \widehat v))],  \nonumber \\
        &(\zeta(\widehat T_\sigma, \widehat v_\sigma)   -   \zeta(\widehat T, \widehat v))(0)  = 0,   
    \end{align}
    for some $\theta \in [0, 1]$.  By the same argument as in proof of the convergence property in \cite[Theorem 2.1]{Casas-2020}, we get 
    \begin{align}
        \|\zeta(\widehat T_\sigma, \widehat v_\sigma)   -   \zeta(\widehat T, \widehat v)\|_{L^\infty(Q_1)}  \to 0 \quad {\rm as} \ \sigma \to 0.
    \end{align}
    By  (\ref{KeyInq22})  of  Lemma \ref{Lemma-stateEq-dd}, we deduce that 
    \begin{align}
         \|\zeta_\sigma(\widehat T_\sigma, \widehat v_\sigma)   -  \zeta(\widehat T_\sigma, \widehat v_\sigma)\|_{L^\infty(Q_1)}  \to 0 \quad {\rm as} \ \sigma \to 0.
    \end{align}
    From these facts, we have 
     \begin{align}
         \|\zeta_\sigma(\widehat T_\sigma, \widehat v_\sigma)   &-  \zeta(\widehat T, \widehat v)\|_{L^\infty(Q_1)}  \nonumber \\ 
         &\le   \|\zeta_\sigma(\widehat T_\sigma, \widehat v_\sigma)   -  \zeta(\widehat T_\sigma, \widehat v_\sigma)\|_{L^\infty(Q_1)} +   \|\zeta(\widehat T_\sigma, \widehat v_\sigma)   -   \zeta(\widehat T, \widehat v)\|_{L^\infty(Q_1)}     \to 0 \quad {\rm as} \ \sigma \to 0.
     \end{align}
  This also implies that $\widehat \zeta_\sigma(1) \to \widehat \zeta(1)$ in $L^\infty(\Omega)$, where $\widehat \zeta_\sigma =  \zeta_\sigma(\widehat T_\sigma, \widehat v_\sigma)$ and $\widehat \zeta = \zeta(\widehat T, \widehat v)$. Hence
  \begin{align}
      H(\widehat \zeta(1))  =  {\rm lim}_{\sigma \to 0} H(\widehat \zeta_\sigma (1)) \le 0. 
  \end{align}
 If $\widehat T = 0$ then we have $\widehat \zeta(x, s) = y_0(x)$ for all $(x, s) \in Q_1$. Hence, $\widehat \zeta(1) = y_0$. Then, we have $H(y_0) = H(\widehat \zeta(1)) \le 0$,  which contradicts $(H2)$. So we must have $\widehat T > 0$. \\
 In a summary, we have showed that there is a subsequence $\{(\widehat T_\sigma, \widehat \zeta_\sigma, \widehat v_\sigma)\}$ of $\{(\bar T_\sigma, \bar \zeta_\sigma, \bar v_\sigma)\}$ such that $\widehat T_\sigma \to \widehat T$ in $\mathbb R$, $\widehat \zeta_\sigma \to \widehat \zeta$ in $L^\infty(Q_1)$ and $\widehat v_\sigma \rightharpoonup  \widehat v$ in $L^p(Q_1)$, where $(\widehat T, \widehat \zeta, \widehat v) \in \Phi''_+$. 
  Moreover, we have 
  \begin{align}
      \widehat T = {\rm lim}_{\sigma \to 0} \widehat T_\sigma \le {\rm lim}_{\sigma \to 0} \widetilde T_\sigma = T_*.
  \end{align}
  Thus $(i)$ is proven.

  $(ii)$ \ Assume that there is $(T_*, \zeta'_*, v'_*) \in \Phi''_+$.  By the condition (\ref{m00}) in Theorem \ref{SOSC-P''}, we have $v'_* = v_*$. It follows from uniqueness of solution of the state equation (\ref{P2.1})-(\ref{P3.1}) that $\zeta'_* = \zeta_*$. Hence, the globally optimal solution $(T_*, \zeta_*, v_*) \in \Phi''_+$ of the problem ${\rm inf}_{(T,  \zeta, v) \in \Phi''_+} T$ is uniquely determined. By a contradiction
argument, we can show that $\bar T_\sigma \to T_*$ in $\mathbb R$, $\bar \zeta_\sigma \to \zeta_*$ in $L^\infty(Q_1)$ and $\bar v_\sigma \rightharpoonup  v_*$ in $L^p(Q_1)$.

Next we claim that for every $1 \le s < +\infty$, $\bar v_\sigma \to  v_*$ in $L^s(Q_1)$.  By similar arguments as \cite[Proposition 4.8]{Bon1},  there exists a sequence $\{T'_\sigma, \zeta', \bar v_\sigma\} \in \Phi''_+$ such that 
\begin{align}
    |T'_\sigma  -  \bar T_\sigma|  \le m_1 |{\rm log} \tau|(\tau + h^2), 
\end{align}
for sufficiently small $\tau$ and $h$ and for some constant  $m_1 > 0$ independent of $\tau$ and $h$. It follows that
\begin{align}
\label{pp1}
    |T'_\sigma  - T_*| \le |T'_\sigma  -  \bar T_\sigma| + |\bar T_\sigma - T_*| \le m_1 |{\rm log} \tau|(\tau + h^2) + |\bar T_\sigma - T_*|  \to 0 \quad {\rm as} \ \sigma \to 0. 
\end{align}
Hence, we may use the growth condition (\ref{m00}) from Theorem \ref{SOSC-P''} to deduce that 
\begin{align}
    \|\bar v_\sigma -  v_*\|_{L^2(Q_1)} \to 0 \quad {\rm as} \ \sigma \to 0.
\end{align}
With  $s \in [1, 2]$, we have $\|\bar v_\sigma -  v_*\|_{L^s(Q_1)}  \le |Q_1|^{\frac{1}{s} - \frac{1}{2}} \|\bar v_\sigma -  v_*\|_{L^2(Q_1)} \to 0$ as $\sigma \to 0$. For $s > 2$, we have
\begin{align}
\label{pp2}
    \|\bar v_\sigma -  v_*\|^s_{L^s(Q_1)}  = \int_{Q_1} |\bar v_\sigma -  v_*|^{s - 2} |\bar v_\sigma -  v_*|^2 dxds \le (b - a)^{s - 2} \|\bar v_\sigma -  v_*\|^2_{L^2(Q_1)} \to 0 \quad {\rm as} \ \sigma \to 0.
\end{align}
So, the claim is justified.

Next we show that estimates (\ref{es1})-(\ref{es4}) are valid. Since $(T_*, \zeta_*, v_*)$ is  optimal solution of the problem ${\rm inf}_{(T,  \zeta, v) \in \Phi''_+} T$, by Proposition \ref{Pro-KKT-P''}, we have  $v_* \in H^1(Q_1)$. From this fact and Lemma \ref{Proj}, 
\begin{align}
        \|I_\sigma v_* - v_*\|_{L^2(Q_1)}  \le m_2(\tau + h),   
    \end{align}
    for some $m_2 > 0$ independent of $\tau$ and $h$. Combining this with similar arguments as \cite[Proposition 4.9]{Bon1}, there exist $\tau_0, h_0 > 0$ such that 
    \begin{align}
        |\bar T_\sigma - T_*| \le c_1 \big(\tau + h  +  |{\rm log} \tau|(\tau + h^2)\big), 
    \end{align}
 for every $\tau \le \tau_0$, $h \le h_0$ and positive constant $c_1 > 0$ independent of $\tau$ and $h$. Hence, the estimate (\ref{es1}) is valid. Moreover, it follows from (\ref{pp1}) that 
 \begin{align}
     |T'_\sigma  - T_*| \le m_3 \big(\tau + h  +  |{\rm log} \tau|(\tau + h^2)\big), 
 \end{align}
 for some $m_3 >  0$, $\tau \le \tau_0$ and $h \le h_0$.  Once again, using the growth condition (\ref{m00}) we get 
 \begin{align}
 \label{pp3}
     \|\bar v_\sigma -  v_*\|_{L^2(Q_1)} \le m_4 \big(\tau + h  +  |{\rm log} \tau|(\tau + h^2)\big)^{\frac{1}{2}},
 \end{align}
for some constants $m_4 >  0$, $\tau \le \tau_0$ and $h \le h_0$.  Also, it follows from (\ref{pp2}) and (\ref{pp3}) that 
\begin{align}
    \|\bar v_\sigma -  v_*\|_{L^s(Q_1)} \le m_5 \big(\tau + h  +  |{\rm log} \tau|(\tau + h^2)\big)^{\frac{1}{s}} \quad \forall s > 2, 
\end{align}
for some constants $m_5 >  0$, $\tau \le \tau_0$ and $h \le h_0$.  Hence, the estimate (\ref{es2}) is valid.  By (\ref{KeyInq11}), (\ref{es1}) and (\ref{es2}), we obtain (\ref{es3}).  By (\ref{KeyInq22}), (\ref{es1}) and (\ref{es2}), we obtain (\ref{es4}).   Therefore, the proof of theorem is complete.
\end{proof}


\begin{corollary}  
    Let the assumptions of  Theorem \ref{main-result}-(ii) hold and $(T_*, \zeta_*, v_*)$ is globally optimal solution of the problem ${\rm inf}_{(T,  \zeta, v) \in \Phi''_+} T$.  Let
    \begin{align*}
        &y_*(x, t) := \zeta_*(x, \frac{t}{T_*}), \quad q_*(x, t) := v_*(x, \frac{t}{T_*}) - \alpha \zeta_*(x, \frac{t}{T_*}), \quad (x, t) \in Q_{T_*}, \\
        &\bar y_\sigma (x, t) := \bar \zeta_\sigma (x, \frac{t}{\bar T_\sigma}), \quad (x, t) \in Q_{\bar T_\sigma}, 
    \end{align*}
    where $Q_{\bar T_\sigma} =  \Omega \times (0, \bar T_\sigma)$. Then $(T_*, y_*, q_*)$ is globally optimal solution of the problem $(P)$ and
        \begin{align}
            &\|\bar y_\sigma - y_{*e}\|_{L^\infty(Q_{\bar T_\sigma})} \to   0   \quad {\rm as}\ \sigma \to 0.  \label{es41} 
        \end{align} 
\end{corollary}
\begin{proof}
    Notice that $\frac{t}{\bar T_\sigma} \in (0, 1)$ for all $t \in (0, \bar T_\sigma)$. By (\ref{es4}) in Theorem \ref{main-result},  for each $(x, t) \in Q_{\bar T_\sigma}$, 
    \begin{align}
        |\bar y_\sigma (x, t) - \zeta_*(x, \frac{t}{\bar T_\sigma})|  =  |\bar \zeta_\sigma (x, \frac{t}{\bar T_\sigma}) - \zeta_*(x, \frac{t}{\bar T_\sigma})|  \le \|\bar \zeta_\sigma - \zeta_*\|_{L^\infty(Q_1)}  \to   0   \quad {\rm as}\ \sigma \to 0,    
    \end{align}
    which implies that 
    \begin{align}
        \|\bar y_\sigma (x, t) - \zeta_*(x, \frac{t}{\bar T_\sigma})\|_{L^\infty(Q_{\bar T_\sigma})} \to   0   \quad {\rm as}\ \sigma \to 0. \label{tt1}
    \end{align}
    Take any $\varepsilon > 0$, we consider two cases:\\
    $\bullet$ If $\bar T_\sigma \le T_*$ then for each $(x, t) \in Q_{\bar T_\sigma}$,  we have 
    \begin{align}
        |\zeta_*(x, \frac{t}{\bar T_\sigma})  -  y_{*e}(x, t)|  = |\zeta_*(x, \frac{t}{\bar T_\sigma})  -  y_{*}(x, t)| 
        = |\zeta_*(x, \frac{t}{\bar T_\sigma})  -  \zeta_*(x, \frac{t}{T_*})|.
    \end{align}
    Since $\frac{t}{\bar T_\sigma} \to \frac{t}{T_*}$ as $\sigma \to 0$ and the uniform continuity of $\zeta_*$ on $\bar \Omega \times [0, 1]$, there exist $\tau', h' > 0$ such that    
     \begin{align}
        &|\zeta_*(x, \frac{t}{\bar T_\sigma})  -  y_{*e}(x, t)|          = |\zeta_*(x, \frac{t}{\bar T_\sigma})  -  \zeta_*(x, \frac{t}{T_*})| < \varepsilon,   \label{tt2}
    \end{align}
    for each $(x, t) \in Q_{\bar T_\sigma}$ with $\tau \le \tau'$ and $h \le h'$.\\
    $\bullet$ If $\bar T_\sigma > T_*$ and  $(x, t) \in \bar \Omega \times [T_*, \bar T_\sigma]$,  we have 
    \begin{align}
        |\zeta_*(x, \frac{t}{\bar T_\sigma})  -  y_{*e}(x, t)|  = |\zeta_*(x, \frac{t}{\bar T_\sigma})  -  y_{*}(x, T_*)| 
        = |\zeta_*(x, \frac{t}{\bar T_\sigma})  -  \zeta_*(x, 1)|.
    \end{align}
    Also $ \frac{T_*}{\bar T_\sigma} \le \frac{t}{\bar T_\sigma} \le 1$ for $t \in [T_*, \bar T_\sigma]$. Hence, $\frac{t}{\bar T_\sigma} \to 1$ as $\sigma \to 0$. From this and the uniform continuity of $\zeta_*$ on $\bar \Omega \times [0, 1]$, there exist $\tau'' \in (0, \tau')$ and $h'' \in (0, h')$ such that  
      \begin{align}
        |\zeta_*(x, \frac{t}{\bar T_\sigma})  -  y_{*e}(x, t)|   
         < \varepsilon,  \label{tt3}
    \end{align}
    for each $(x, t) \in Q_{\bar T_\sigma}$ with $\tau \le \tau''$ and $h \le h''$.     From (\ref{tt2}) and (\ref{tt3}), we deduce that 
    \begin{align}
        \|\zeta_*(x, \frac{t}{\bar T_\sigma})  -  y_{*e}(x, t)\|_{L^\infty(Q_{\bar T_\sigma})} \to   0   \quad {\rm as}\ \sigma \to 0. \label{tt4}
    \end{align}
    Combining (\ref{tt1}) and (\ref{tt4}), we have 
    \begin{align}
        \|\bar y_\sigma - y_{*e}\|_{L^\infty(Q_{\bar T_\sigma})} \le  \|\bar y_\sigma (x, t) - \zeta_*(x, \frac{t}{\bar T_\sigma})\|_{L^\infty(Q_{\bar T_\sigma})} + \|\zeta_*(x, \frac{t}{\bar T_\sigma})  -  y_{*e}(x, t)\|_{L^\infty(Q_{\bar T_\sigma})} \to 0
    \end{align}
    as $\sigma \to 0$. The proof of the corollary is complete.
\end{proof}

\subsection{Numerical example}

This subsection is dedicated to a numerical example in which we compute approximate solutions to a time-optimal control problem. 

\begin{example} {\rm  
Let $\Omega = (0, 1)^2 = \{(x_1, x_2) : 0 <  x_1, x_2 < 1\}$ and 
\begin{align*}
    &y_0(x_1, x_2) = 27 x_1(1 - x_1)\sin(\pi x_2), \quad     y_\Omega(x_1, x_2) = 0, \\
    &\lambda = 0.1,  \quad     H(y(T)) = \frac{1}{2}\int_\Omega y^2(T) dx  -   \frac{\lambda^2}{2}. 
\end{align*}
Consider the problem
\begin{align}
\label{EX1}
    {\rm {Minimize}} \ T \quad {\rm {subject  \ to}} \quad 
    \begin{cases}
        (T, y, w) \in (0, +\infty)  \times \big(C(\bar Q_T)\cap W^{1, 1}_2(0, T; D, H)\big) \times L^\infty(Q_T),  \\
        \frac{\partial y}{\partial t}  -   \Delta y   = w \quad \text{in} \ Q_T, \quad y(x, t)=0 \quad \text{on}  \ \Sigma_T,  \\
        y(0)=y_0  \quad \text{in} \ \Omega,   \\
        H(y(T))  \le  0,  \\
        - 1  \le w(x, t) + \dfrac{1}{5} y(x, t) \le 3 \quad \text{a.a.} \ (x, t)\in Q_T.  
    \end{cases}
\end{align}
It is easy to check that the problem satisfies assumptions $(H1)$-$(H2)$ and so it has at least one globally optimal solution. By changing control variable $u  = w + \dfrac{1}{5} y$, the problem (\ref{EX1}) can be transformed to time-optimal control problem with pure control constraint:
\begin{align}
\label{EX2}
    {\rm {Minimize}} \ T \quad {\rm {subject  \ to}} \quad 
    \begin{cases}
        \frac{\partial y}{\partial t}  +  Ay   =  u \quad \text{in} \ Q_T, \quad y(x, t)=0 \quad \text{on}  \ \Sigma_T,  \\
        y(0)=y_0  \quad \text{in} \ \Omega,   \\
        H(y(T))  \le  0,  \\
        - 1  \le u(x, t)\le 3 \quad \text{a.a.} \ (x, t)\in Q_T,   
    \end{cases}
\end{align}
where $Ay = - \Delta y + \frac{1}{5}y$. Based on the equivalence of time- and distance-optimal control problems (see \cite{Bon2}), finding an approximate solution to (\ref{EX2}) leads to a bi-level optimization problem:  in the outer loop, the optimal time is determined by a Newton method as the solution of the optimal value function (see Figure 3.1 for our example); and with each given $T$, we determine a control such that the associated state has a minimal distance to the target set. For further details we refer to \cite[Algorithm 1]{Bon2}.
{\begin{center}
\begin{tikzpicture}
\small
\begin{axis}[%
width=2.521in,
height=2.906in,
at={(0.758in,0.481in)},
scale only axis,
xmin=0,
xmax=0.4,
xlabel style={font=\color{white!15!black}},
xlabel={$T$},
ymin=-0.5,
ymax=3.5,
ylabel style={font=\color{white!15!black}},
ylabel={$\delta(T)$},
axis background/.style={fill=white}
]
\addplot [color=white!30!black, dotted, line width=1.3pt, forget plot]
  table[row sep=crcr]{%
0.002	0\\
0.4	0\\
};
\addplot [color=mycolor1, line width=1.3pt, forget plot]
  table[row sep=crcr]{%
0.002	3.24249915737058\\
0.0122051282051282	2.6177635825842\\
0.0224102564102564	2.10944139792288\\
0.0326153846153846	1.69514685594555\\
0.0428205128205128	1.35734656403721\\
0.053025641025641	1.08186648765007\\
0.0632307692307692	0.857178340449307\\
0.0734358974358974	0.673893676838214\\
0.0836410256410256	0.524364076492223\\
0.0938461538461539	0.402357646612884\\
0.104051282051282	0.302796031143306\\
0.11425641025641	0.221540403970732\\
0.124461538461538	0.155217285413091\\
0.134666666666667	0.101076768230246\\
0.144871794871795	0.0568771398608765\\
0.155076923076923	0.0207910385131232\\
0.165282051282051	-0.00867074216396987\\
0.175487179487179	-0.0327207665963598\\
0.185692307692308	-0.052343576650874\\
0.195897435897436	-0.0683303772286512\\
0.206102564102564	-0.0812884634480287\\
0.216307692307692	-0.0915529726580108\\
0.226512820512821	-0.0986853514249463\\
0.236717948717949	-0.0999919095323038\\
0.246923076923077	-0.0999915851378888\\
0.257128205128205	-0.0999934107665557\\
0.267333333333333	-0.0999910972912562\\
0.277538461538462	-0.0999928014817281\\
0.28774358974359	-0.0999941641426056\\
0.297948717948718	-0.0999951754932994\\
0.308153846153846	-0.0999946256730745\\
0.318358974358974	-0.0999936144387468\\
0.328564102564103	-0.0999946212322081\\
0.338769230769231	-0.0999957014335574\\
0.348974358974359	-0.0999964724242012\\
0.359179487179487	-0.0999971117370307\\
0.369384615384615	-0.0999975992316652\\
0.379589743589744	-0.0999980373985562\\
0.389794871794872	-0.0999971103641475\\
0.4	-0.0999976270115861\\
};
\end{axis}
\end{tikzpicture}\\
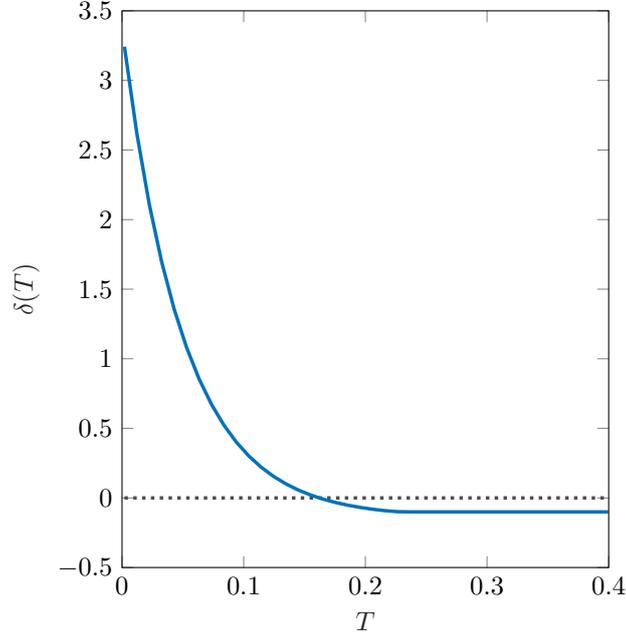
{\footnotesize {\textsc{Figure 3.1}. The minimum distance value function $\delta(T)$ associated with the problem (\ref{EX2}).}}
\end{center}}

\begin{table}[ht]
\centering
\small
\begin{tabular}{cccccccl}
\toprule
$M$ & $N$ & $T_\sigma$ & $|T_\sigma - T_{\sigma_*}|$ & $EOC$ & Newton steps & D-steps & cG steps \\
\toprule 
20 & 4225 & 0.175221 & $1.3570 \times 10^{-2}$ & $--$ & 5 & 0 & 2 \\
40 & 4225 & 0.168169 & $6.5187 \times 10^{-3}$ & 1.06 & 5 & 0 & 2 \\
80 & 4225 & 0.164784 & $3.1333 \times 10^{-3}$ & 1.06 & 5 & 0 & 2 \\
160 & 4225 & 0.163125 & $1.4745 \times 10^{-3}$ & 1.09 & 5 & 0 & 2 \\
320 & 4225 & 0.162304 & $6.5345 \times 10^{-4}$ & 1.17 & 5 & 0 & 2 \\
640 & 4225 & 0.161896 & $2.4499 \times 10^{-4}$ & 1.42 & 5 & 0 & 2 \\
\midrule
1280 & 25 & 0.144915 & $1.6736 \times 10^{-2}$ & $--$ & 5 & 0 & 2 \\
1280 & 81 & 0.156844 & $4.8064 \times 10^{-3}$ & 1.80 & 5 & 0 & 2 \\
1280 & 289 & 0.160483 & $1.1681 \times 10^{-3}$ & 2.04 & 5 & 0 & 2 \\
1280 & 1089 & 0.161447 & $2.0385 \times 10^{-4}$ & 2.52 & 5 & 0 & 2 \\
1280 & 4225 & 0.161692 & $4.1267 \times 10^{-5}$ & 2.31 & 5 & 0 & 2 \\
\bottomrule
\end{tabular}
\label{tab:convergence}

\vspace{0.9em}

\begin{minipage}{\textwidth}
{\footnotesize {\noindent {\sc Table} 3.1. Results of the convergence test for Example 3.1, showing computed optimal times, absolute errors, Experimental Orders of Convergence (EOC) in space and time, number of Newton steps in the outer loop, and iteration steps used in the numerical method, with $M$ denoting the number of time steps and $N$ the number of nodes for the spatial discretization. The initial value for the Newton method is $T_0 = 0.1$.}}
\end{minipage}
\end{table}

To discretize the problem, we use meshes that are uniform in both time and space, with size $\sigma = (\tau, h)$. The state and adjoint state equations are discretized by means of piecewise constant functions in time (corresponding to the implicit Euler method) and continuous and cellwise linear functions in space. Since we do not know the exact optimal time, we calculate a numerical reference time on an additionally refined grid with the number of time steps $M = 2560$ and the number of nodes for the spatial discretization $N = 16641$. The optimal time we obtain numerically is approximately $T_{\rm opt} \approx T_{\sigma_*} = 0.1616506$.

We perform two experiments: In the first one, we fix the number of nodes for the spatial discretization $N = 4225$ and we measure the error $|T_\sigma - T_{\sigma_*}|$ as the number of time steps $M$ increase. The Experimental Orders of Convergence in time are defined as follows:
\begin{align*}
    \frac{\log (|T_{\sigma_{i+1}} - T_{\sigma_*}|) -  \log (|T_{\sigma_{i}} - T_{\sigma_*}|)}{\log \tau_{i + 1} - \log \tau_i},
\end{align*}
where $\tau_i = \frac{1}{M_i}$.  
In the second one, we fix $M = 1280$ and measure the error $|T_\sigma - T_{\sigma_*}|$ as the space mesh size varies. The Experimental Orders of Convergence in space are computed by
\begin{align*}
    \frac{\log (|T_{\sigma_{i+1}} - T_{\sigma_*}|) -  \log (|T_{\sigma_{i}} - T_{\sigma_*}|)}{\log h_{i + 1} - \log h_i},
\end{align*}
where $h_i = \frac{\sqrt{2}}{\sqrt{N_i} - 1}$.  
All experiments were computed using MATLAB R2024b on Windows 11 (64-bit), 13th Gen Intel(R) Core(TM) i7-1355U 1.70 GHz, 16.0 GB RAM.  The results are shown in Table 3.1. We can see that the EOC in time is ${\mathcal O}(\tau)$, while the EOC in space is ${\mathcal O}(h^2)$, which are quite in agreement, up to logarithmic terms, with the theoretical result given in Theorem~\ref{main-result}.
}    
\end{example}

\medskip

\noindent {\bf Acknowledgments.}   This research was supported by Vietnam Academy of Science and Technology under grant number CTTH00.01/25-26.


\begin{thebibliography}{99}








\bibitem{Arada-2003} N. Arada and J.-P. Raymond,  {\it Time optimal problems with Dirichlet boundary controls},  Discrete and Continuous Dynamical Systems, 2003, 9(6): 1549-1570.







\bibitem{Bon} L. Bonifacius, K. Pieper and B. Vexler, {\it A priori estimates for space time-time finite element discretization of parabolic time-optimal problems}, SIAM J. Control Optim.,  57(2019),  129-162.

\bibitem{Bon1} L. Bonifacius, K. Pieper and B. Vexler,  {\it  Error estimates for space-time discretization of parabolic time-optimal control problems with bang-bang controls}, SIAM J. Control Optim., 57 (2019), 1730-1756.

\bibitem{Bon2} L. Bonifacius, K. Kunisch,  {\it Time-optimality by distance-optimality for parabolic control systems}, ESAIM: M2AN 54 (2020) 79-103. 



\bibitem{0.7} S.C. Brenner, L.R. Scott,  {\it The Mathematical Theory of Finite Element Methods}, Texts in Applied Mathematics, vol. 15, 2nd edn. Springer, New York (2002). https://doi.org/10.1007/978-1-4757-3658-8






\bibitem{CARSTENSEN}  C. Carstensen,
{\it Quasi-interpolation and a posteriori error analysis in finite element methods}, M2AN Math. Model. Numer. Anal., 33(6):1187-1202, 1999.


\bibitem{Casas-2018} E. Casas, M. Mateos, and A. R\"{o}sch,  {\it Improved approximation rates for a parabolic control problem with an objective promoting directional sparsity}. Comput Optim Appl 70, 239-266 (2018). https://doi.org/10.1007/s10589-018-9979-0


\bibitem{Casas-2020} E. Casas and M. Mateos, {\it Critical cones for sufficient second order conditions in PDE constrained optimization}, SIAM J. Optim., 30(2020),  585-603. 






\bibitem{Casas-2023} E. Casas and D. Wachsmuth, {\it A Note on existence of solutions to control problems of semilinear partial differential equations}, SIAM J. Control Optim., 61 (2023) 1095-1112. 

\bibitem{Casas-2023-2} E. Casas, K. Kunisch, M. Mateos,  {\it Error estimates for the numerical approximation of optimal control problems with nonsmooth pointwise-integral control constraints}. IMA Journal of Numerical Analysis, Volume 43, Issue 3, May 2023, Pages 1485$-$1518, https://doi.org/10.1093/imanum/drac027


\bibitem{Evan} L.C. Evan, {\it Partial Differential Equations}, AMS, Providence Rhode Island, 2010. 

\bibitem{Grisvard} P. Grisvard, {\it Elliptic Problems in Nonsmooth Domains}, Pitman Publishing Inc., 1985.








\bibitem{KKR-2024-3}  H. Khanh, B.T. Kien and  A. R\"{o}sch, {\it Second-order Optimality Conditions for Time-Optimal Control Problems Governed by Semilinear Parabolic Equations}, https://arxiv.org/abs/2411.07723  (2024). 












\bibitem{Neitzel-2012} I. Neitzel,  B. Vexler, {\it A priori error estimates for space–time finite element discretization of semilinear parabolic optimal control problems}. Numer. Math. 120, 345-386 (2012).













\end{thebibliography}
\end{document}